\documentclass[11pt, a4paper]{article}

\usepackage[utf8]{inputenc}
\usepackage[T1]{fontenc}
\usepackage{lmodern}
\usepackage{amssymb, amsthm, amsfonts, mathrsfs}
\usepackage{geometry}
\usepackage{mathtools}
\usepackage[all]{xy}
\usepackage{graphicx}
\usepackage{microtype}
\usepackage{enumitem}
\usepackage{stmaryrd}
\usepackage{amscd}
\usepackage[hidelinks]{hyperref}

\theoremstyle{plain}
\newtheorem{theorem}{Theorem}[section]
\newtheorem{lemma}[theorem]{Lemma}
\newtheorem{proposition}[theorem]{Proposition}
\newtheorem{corollary}[theorem]{Corollary}
\newtheorem*{proposition*}{Proposition}
\theoremstyle{definition}
\newtheorem{definition}[theorem]{Definition}

\newtheorem{notation}[theorem]{Notation}

\theoremstyle{remark}
\newtheorem{remark}[theorem]{Remark}
\newtheorem{construction}[theorem]{Construction}

\newtheorem{claim}[theorem]{Claim}
\newtheorem{convention}[theorem]{Convention}
\newtheorem*{theorem*}{Theorem}

\newcommand{\Z}{\mathbb{Z}}
\newcommand{\R}{\mathbb{R}}
\newcommand{\C}{\mathbb{C}}
\newcommand{\N}{\mathbb{N}}
\newcommand{\Q}{\mathbb{Q}}

\newcommand{\Clf}{C^{lf}_{\infty}(\C)}
\newcommand{\Conf}{Conf^{lf}_{\infty}(\C)}

\newcommand{\HoQuot}{/\!/}

\DeclareMathOperator{\Aut}{Aut}
\DeclareMathOperator{\supp}{supp}
\DeclareMathOperator{\dist}{dist}
\DeclareMathOperator{\id}{id}

\DeclareMathOperator{\Sym}{Sym}
\DeclareMathOperator{\pr}{pr}
\DeclareMathOperator{\Lip}{Lip}

\title{Topological fundamental groups of locally finite infinite configuration spaces and infinite braids}

\author{Jyh-Haur Teh\thanks{Department of Mathematics, National Tsing Hua University, Taiwan.
E-mail: \texttt{jyhhaur@math.nthu.edu.tw}.}}
\date{} 

\begin{document}

\maketitle

\begin{abstract}
We study the topological fundamental groups of the locally finite infinite ordered configuration space \(Conf^{lf}_\infty(\C)\) in the plane and the homotopy quotient of $Conf^{lf}_\infty$ by the canonical action of the infinite permutation group $\Aut(\N)$:
\[
H^{lf}(\infty):=\pi_1^{\mathrm{top}}(Conf^{lf}_\infty(\C),\widetilde{\N}),
\qquad
B^{lf}(\infty):=\pi_1^{\mathrm{top}}\!\bigl(Conf^{lf}_\infty(\C)\!/\!/\Aut(\N),[e_0,\widetilde{\N}]\bigr).
\]
We prove that \(H^{lf}(\infty)\) and \(B^{lf}(\infty)\) are non-discrete and complete topological groups.  A main structural theorem identifies \(H^{lf}(\infty)\) with a canonical locally finite inverse-limit model built from finite pure braid groups, and we construct a complete left-invariant ultrametric compatible with the quotient topology from the loop space of $\Conf$. The direct limit of finite pure braid groups admits a dense embedding into \(H^{lf}(\infty)\), and we show that
\(H^{lf}(\infty)\) is the Ra\u{\i}kov completion of this subgroup. Moreover, the direct limit of finite braid groups
embeds into \(B^{lf}(\infty)\) and is dense in the finitary subgroup \(B^{lf}_{\mathrm{fin}}(\infty)\subseteq
B^{lf}(\infty)\).

\end{abstract}

\section{Introduction}

Configuration spaces sit at a classical interface of topology, geometry, and analysis.  For a manifold $M$
and $n\ge1$, the ordered configuration space
\[
Conf_n(M)=\{(x_1,\dots,x_n)\in M^n \mid x_i\neq x_j\text{ for }i\neq j\}
\]
and its unordered quotient $C_n(M)=Conf_n(M)/S_n$ encode, for surfaces, the pure braid groups and braid groups.
The foundational fibrations of Fadell--Neuwirth (\cite{FadellNeuwirth62}, \cite{FoxNeuwirth62}, \cite{Arnold69}) and the subsequent developments of McDuff, Segal, and others (e.g.\ \cite{McDuff1975, Segal1973}) place these spaces among the basic objects in
algebraic topology; see (\cite{Kallel24}, \cite{KasselTuraev08}) for a broad guide and further references.

\medskip
\noindent\textbf{Locally finite infinite configurations.}
In \cite{Teh25}, we introduced and studied a genuinely ``infinite, locally finite'' planar configuration
space, designed to model countably many particles escaping to infinity while retaining a topology sensitive to
local statistics.  Concretely, let $Conf^{lf}_\infty(\C)$ be the space of \emph{locally finite ordered}
configurations $\widetilde A=(a_1,a_2,\dots)$ in $\C$ (no repetitions and no accumulation in compact subsets),
and let $C^{lf}_\infty(\C)$ be the corresponding \emph{unordered} space, obtained by the quotient
$Conf^{lf}_\infty(\C)/\Aut(\N)$.  The topology on $C^{lf}_\infty(\C)$ is the one naturally induced by the
\emph{vague topology} on simple counting measures, and we equip $Conf^{lf}_\infty(\C)$ with a compatible metric
(the sum metric from \cite{Teh25}) making the projection
\[
P:Conf^{lf}_\infty(\C)\longrightarrow C^{lf}_\infty(\C)
\]
continuous.

Two features distinguish this setting from the direct limit $\varinjlim_n Conf_n(\C)$.  First, the configuration
is truly infinite at every stage; local finiteness prevents compact accumulation but does not truncate the tail.
Second, the resulting fundamental groups are dramatically larger: for instance, $\pi_1\!\bigl(Conf^{lf}_\infty(\C)\bigr)$
is uncountable, whereas $\pi_1(\varinjlim_n Conf_n(\C))$ is countable.

\medskip
\noindent\textbf{From homotopy type to topological group structure.}
In $\cite{Teh25}$ we showed that $Conf^{lf}_\infty(\C)$ and its homotopy quotient $Conf^{lf}_\infty(\C)//\Aut(\N)$
are spherical. The present paper continues the work of \cite{Teh25}, but shifts emphasis from the homotopy type of
$Conf^{lf}_\infty(\C)$ and $C^{lf}_\infty(\C)$ to the \emph{canonical topological group structures} carried on the
associated locally finite braid groups.  Write
\[
H^{lf}(\infty):=\pi_1\!\bigl(Conf^{lf}_\infty(\C),\widetilde\N\bigr),
\qquad
B^{lf}(\infty):=\pi_1\!\bigl(Conf^{lf}_\infty(\C)//\Aut(\N),\widetilde\N\bigr),
\]
where $//$ denotes the Borel (homotopy) quotient and $\widetilde\N=(1,2,3,\dots)$ is the standard basepoint.
As abstract groups, \cite{Teh25} established a locally finite analogue of the braid exact sequence
\begin{equation}\label{eq:lf-braid-seq-intro}
1\longrightarrow H^{lf}(\infty)\longrightarrow B^{lf}(\infty)\longrightarrow \Aut(\N)\longrightarrow 1,
\end{equation}
with $\Aut(\N)$ regarded as discrete.

In this paper we equip $H^{lf}(\infty)$ and $B^{lf}(\infty)$ with the \emph{quotient topology} induced from the
based loop spaces with the uniform topology coming from the ambient metrics on configuration spaces.  This is
the most natural topology one can place on $\pi_1$ once one decides that the ambient configuration spaces are
the primary geometric objects.  The resulting topological groups are not at all analogous to classical braid
groups: in the finite-dimensional manifold setting, the same construction yields a discrete group.  Here the
locally finite topology forces a genuinely \emph{non-discrete}, infinite-dimensional phenomenon.
Topological fundamental groups have been studied for decades (see \cite{Dugundji50}). More recent work on
the topological fundamental group of configuration spaces can be found in \cite{Fabel05, Fabel06, Brazas11, Brazas13, BrazasFabel15}.

\medskip
\noindent\textbf{Main result: completeness and non-discreteness.}
Our main theorems shows that the locally finite braid groups are not merely large abstract groups: they carry a robust, analytically meaningful topology. Equipped with the quotient topology inherited from the based loop space, the corresponding fundamental groups are, in general, only quasitopological, rather than genuinely topological groups. We discover that $H^{lf}(\infty)$ and $B^{lf}(\infty)$ posses several amazing properties. We list some of them:
\begin{itemize}
\item Both $H^{lf}(\infty)$ and $B^{lf}(\infty)$ are completely metrizable topological groups. In fact, each admits a left-invariant ultrametric that induces its topology, and every Cauchy sequence with respect to this metric converges.
\item Neither $H^{lf}(\infty)$ nor $B^{lf}(\infty)$ is discrete.
\item $H^{lf}(\infty)$ is the completion of the direct limit of finite pure braid groups.
\item A clopen subgroup $B^{lf}_{\mathrm{fin}}(\infty)$ of $B^{lf}(\infty)$ is the completion of the direct limit of finite braid groups.
\item $H^{lf}(\infty)$ and $B^{lf}_{\mathrm{fin}}$ are Polish spaces.
\end{itemize}

The non-discreteness is already visible at the geometric level: one can construct nontrivial braids supported
far out in the configuration near infinity whose representatives are arbitrarily close to the constant loop
in the uniform metric, hence converge to the identity element in $\pi_1$.
In other words, ``small'' in the locally finite topology does \emph{not} imply ``trivial'' in the braid group.
This phenomenon is impossible for the classical braid groups coming from $Conf_n(\C)$, and it is precisely the
tail of the locally finite configuration that creates room for such vanishing-at-infinity braiding.

Completeness, by contrast, is a structural statement: Cauchy families of locally finite braids admit limits,
so the resulting groups are closed under passage to infinite limiting braid patterns controlled at infinity.
From the perspective of topological group theory, completeness is the minimum regularity needed to bring
standard tools (Baire category methods, automatic continuity questions, and continuity of actions) to bear.

\medskip
\noindent\textbf{Exactness in the category of topological groups.}
A second theme is that the locally finite braid sequence is naturally compatible with the topologies.

\begin{theorem*}[Topological exact sequence]
The sequence \eqref{eq:lf-braid-seq-intro} is exact as a sequence of \emph{topological groups}: the inclusion
$H^{lf}(\infty)\hookrightarrow B^{lf}(\infty)$ is a topological embedding onto a closed normal subgroup, and
the quotient topology on $B^{lf}(\infty)/H^{lf}(\infty)$ identifies it with the discrete group $\Aut(\N)$.
\end{theorem*}

Thus $B^{lf}(\infty)$ should be viewed as a topological extension of a non-discrete complete group
by a discrete symmetry group.  This sharply separates two sources of complexity: continuous locally finite
braiding (the kernel) versus discrete relabeling at infinity (the quotient).

\medskip
\noindent\textbf{Context and motivation.}
The locally finite configuration spaces $Conf^{lf}_\infty(\C)$ and $C^{lf}_\infty(\C)$ arise naturally beyond
classical braid theory.  As explained in \cite{Teh25}, $C^{lf}_\infty(\C)$ is the standard state
space of simple point configurations in stochastic geometry and point process theory when endowed with the
vague topology, and it is also the natural recipient of zero sets of families of entire functions via
Weierstrass products.  The present topological-group refinement is motivated by the observation that many
applications (in particular, those involving parametrized families, limiting procedures, or dynamics) do not
only require $\pi_1$ as an abstract invariant: they require a topology on $\pi_1$ compatible with the topology
of the underlying configuration space.

In this sense, the completeness and non-discreteness of $H^{lf}(\infty)$ and $B^{lf}(\infty)$ are not
decorative properties.  They encode a new kind of ``braid-at-infinity'' continuity which is invisible to
finite configuration spaces and which is stable under limiting operations intrinsic to locally finite topology.

\medskip
\noindent\textbf{Organization of the paper.}
In section 2, we recall the definition of topological fundamental group $\pi_1^{top}(X,x_0)$, the quotient of the based loop space $\Omega(X,x_0)$ by based homotopy. We show that finite pure braid groups $P_n=\pi_1^{top}(Conf_n(\mathbb C))$ are discrete,
and construct, for every $\varepsilon>0$, an $\varepsilon$-small but essential loop in $(Conf,d_{\Sigma})$
by braiding only a strand far out at infinity. This shows that
$\Conf$ is not semilocally simply connected and $H^{lf}(\infty)$ is \emph{non-discrete}.

In section 3, we define the inverse limit $\varprojlim P_n$ with forgetful maps $p_{m,n}$ and prove it is completely metrizable.
The truncation projection $\Conf\to Conf_n(\mathbb C)$ is a Hurewicz fibration, enabling fibration-based arguments about lifting homotopies and controlling truncations.

In section 4, we introduce the key refinement:
$
(\varprojlim_n P_n)_{lf}\subseteq \varprojlim_n P_n,
$
the subgroup of compatible braid data realizable by loops whose ``extra strands'' eventually stay outside any fixed compact set
and show that $(\varprojlim P_n)_{lf}$ is a \emph{proper} subgroup of $\varprojlim_n P_n$.

In section 5, we define the canonical map
$
\Theta:H^{lf}(\infty)\to (\varprojlim_n P_n)_{lf}
$
and prove that it is a continuous bijective group homomorphism.

In section 6, we define open normal subgroups
$
V_n:=\ker\bigl(p_n:H^{lf}(\infty)\to P_n\bigr),
$
and prove $\{V_n\}$ is a neighborhood basis at the identity with $\bigcap_n V_n=\{1\}$.
We show that any $g\in V_n$ has representatives that literally fix the first $n$ strands and keep remaining strands outside a prescribed compact. These yield joint continuity of multiplication, so $H^{lf}(\infty)$ is a \emph{topological group}.
Comparing the cylinder neighborhoods in $(\varprojlim P_n)_{lf}$ with the $V_n$ basis upgrades $\Theta$ to a topological group isomorphism.
We define an explicit left-invariant ultrametric on $H^{lf}(\infty)$ which is compatible and complete.

In section 7, the completeness of $B^{lf}(\infty)$ is proved. We show that
components of $\Omega(BG)$ are contractible and $\pi_1^{top}(BG)\cong G$ is discrete.
The strategy is to view $B^{lf}(\infty)$ as a topological group extension with kernel $H^{lf}(\infty)$
and discrete quotient, and to deduce that $B^{lf}(\infty)$ is complete.

In section 8, we realize finite braid groups via fundamental groups of Borel constructions, recovering the standard exact sequence
$1\to P_n\to B_n\to \Sigma_n\to 1$ and discreteness at finite level. We
constructs compatible inclusions of finite configuration spaces (and braid groups) into the infinite locally finite setting, producing subgroups $P_\infty\le H^{lf}(\infty)$ and $B_\infty\le B^{lf}(\infty)$.
We prove the “pure” direct limit $P_\infty$ is dense in $H^{lf}(\infty)$, while the “full” direct limit $B_\infty$ is not dense; its closure is exactly the preimage of finitary permutations:
$
\overline{B_\infty}=\pi_*^{-1}(\Sym_f(\mathbb N)),
$
a clopen subgroup of $B^{lf}(\infty)$.

In section 9, we introduce Ra\u{\i}kov completeness and Polishness. We show that $H^{lf}(\infty)$ is Ra\u{\i}kov complete and is the Ra\u{\i}kov completion of the dense subgroup $P_\infty$. We show that the finitary full subgroup
$
B^{lf}_{\mathrm{fin}}(\infty):=\pi_*^{-1}(\Sym_f(\mathbb N)),
$
and $B^{lf}(\infty)$ are Ra\u{\i}kov complete, and identify $B^{lf}_{\mathrm{fin}}$ as the Ra\u{\i}kov completion of $B_\infty$.
Furthermore, we show that $H^{lf}(\infty)$ is Polish; $B^{lf}(\infty)$ is not Polish; but $B^{lf}_{\mathrm{fin}}(\infty)$ is Polish.

\section{Topological fundamental groups and completeness}

\subsection{Topological fundamental groups and completely metrizable spaces}

\begin{definition}\label{def:pi1top}
Let $(X,x_0)$ be a based space. Write $\Omega(X,x_0)$ for the based loop space with the compact--open topology.
Define the \emph{topological fundamental group} $\pi_1^{\mathrm{top}}(X,x_0)$ to be the set $\pi_1(X,x_0)$
equipped with the quotient topology induced by the canonical surjection
\[
q_X:\Omega(X,x_0)\twoheadrightarrow \pi_1(X,x_0),\qquad \gamma\mapsto[\gamma].
\]
\end{definition}

\begin{remark}
In general $\pi_1^{\mathrm{top}}(X,x_0)$ need \emph{not} be a topological group:
multiplication may fail to be jointly continuous. What is always true is the weaker
``quasitopological'' property recorded below.
\end{remark}

\begin{lemma}[Quasitopological group structure]\label{lem:qtop}
For any based space $(X,x_0)$, inversion
$[\gamma]\mapsto[\gamma^{-1}]$ and left/right translations
$L_{[\alpha]}([\gamma])=[\alpha\cdot\gamma]$ and $R_{[\alpha]}([\gamma])=[\gamma\cdot\alpha]$
are continuous maps on $\pi_1^{\mathrm{top}}(X,x_0)$. In particular, each translation is a homeomorphism.
\end{lemma}

\begin{proof}
Concatenation with a \emph{fixed} loop is continuous on $\Omega(X,x_0)$ in the compact--open topology, and so is
reversal $\gamma\mapsto\gamma^{-1}$. These maps respect based homotopy, hence descend to continuous maps on the
quotient $\pi_1^{\mathrm{top}}(X,x_0)$. Since $L_{[\alpha]}^{-1}=L_{[\alpha^{-1}]}$ and
$R_{[\alpha]}^{-1}=R_{[\alpha^{-1}]}$, translations are homeomorphisms.
\end{proof}

\begin{definition}[Complete metrizability]\label{def:complete-metrizable}
A topological space is \emph{completely metrizable} if its topology is induced by some complete metric.
A topological group is \emph{a complete topological group} if it is a topological group whose underlying space
is completely metrizable.
\end{definition}

\subsection{The locally finite pure braid group $H^{lf}(\infty)$}

\subsubsection{The vague metric and the sum metric}

Let $C_c(\C)$ denote the space of real-valued continuous functions on $\C$ with compact support.
For $r>0$, let
\[
K_r:=\{z\in\C:\ |z|\le r\}=\overline{B}(0,r).
\]

We need a basic construction of a family of continuous, compactly supported functions in order to define the vague metric.
The following definition is equivalent to that in \cite{Teh25}, but is better suited to the analysis carried out later in this paper.

\begin{lemma}\label{lem:testfamily-strong}
There exists a sequence $(\varphi_j)_{j\ge1}\subset C_c(\C)$ such that:
\begin{enumerate}[label=\textnormal{(\alph*)}]
\item For every $m\in\N$, the subfamily $\{\varphi_j:\ \supp(\varphi_j)\subset K_m\}$ is dense in
\[
C_c(K_m):=\{\psi\in C_c(\C):\ \supp(\psi)\subset K_m\}
\]
with respect to $\|\cdot\|_\infty$.
\item For every $q\in\Q(i)$ and every $0<r_1<r_2$ in $\Q$, there exists $j$ such that
\[
0\le \varphi_j\le 1,\qquad
\varphi_j\equiv 1 \text{ on } \overline{B}(q,r_1),\qquad
\supp(\varphi_j)\subset B(q,r_2).
\]
\item For every $j\ge1$, one has $\supp(\varphi_j)\subset K_j$.
\end{enumerate}
\end{lemma}

\begin{proof}
Fix $m\in\N$. The Banach space $C(K_m)$ is separable. The subspace
\[
C_c(K_m)\;=\;\{\psi\in C(K_m):\ \psi|_{\partial K_m}=0\}
\]
is closed in $C(K_m)$ (hence separable). Choose a countable $\|\cdot\|_\infty$--dense family
$\{f_{m,k}\}_{k\ge1}\subset C_c(K_m)$.

Next, for each triple $(q,r_1,r_2)$ with $q\in\Q(i)$ and $0<r_1<r_2$ in $\Q$, choose a continuous function
$\eta_{q,r_1,r_2}:\C\to[0,1]$ such that
\[
\eta_{q,r_1,r_2}\equiv 1 \text{ on } \overline{B}(q,r_1),
\qquad
\supp(\eta_{q,r_1,r_2})\subset B(q,r_2).
\]

For each $m$, define a countable set
\[
S_m:=\{f_{m,k}:k\ge1\}\ \cup\ \{\eta_{q,r_1,r_2}:\ \supp(\eta_{q,r_1,r_2})\subset K_m\}\ \subset\ C_c(K_m),
\]
and enumerate $S_m=\{\psi_{m,k}\}_{k\ge1}$.

Now define a single sequence $(\varphi_j)_{j\ge1}$ by placing $\psi_{m,k}$ at indices
\[
j(m,k):=2^m(2k-1),
\qquad
\varphi_{j(m,k)}:=\psi_{m,k},
\]
and set $\varphi_j\equiv 0$ for indices $j$ not of the form $j(m,k)$.

Then $j(m,k)\ge 2^m\ge m$ and $\supp(\psi_{m,k})\subset K_m\subset K_{j(m,k)}$, and also $\supp(0)=\varnothing\subset K_j$.
Hence (c) holds. For fixed $m$, the subfamily $\{\varphi_j:\supp(\varphi_j)\subset K_m\}$ contains $\{f_{m,k}\}_{k\ge1}$,
so it is dense in $C_c(K_m)$; this proves (a). Property (b) holds because for any $(q,r_1,r_2)$ we can choose $m$ large enough
that $\overline{B}(q,r_2)\subset K_m$, so $\eta_{q,r_1,r_2}\in S_m$ and therefore appears among the $\varphi_j$.
\end{proof}

Throughout this paper, fix $\{\varphi_j\}_{j\ge1}$ as in Lemma~\ref{lem:testfamily-strong}.

\begin{definition}[The vague metric $d_{\mathcal V}$]\label{def:vague-metric}
Define $d_{\mathcal{V}}$ on $\Clf$ by
\[
d_{\mathcal{V}}(A, B) = \sum_{j=1}^{\infty} \frac{1}{2^j} \,
\frac{\left| \sum_{a \in A} \varphi_j(a) - \sum_{b \in B} \varphi_j(b) \right|}{1 + \left| \sum_{a \in A} \varphi_j(a) - \sum_{b \in B} \varphi_j(b) \right|}.
\]
\end{definition}

Let $P: \Conf \rightarrow \Clf$ be the projection. Equip $\Conf$ with the sum metric
\begin{equation}\label{eq:d-sum}
d_{\sum}(x,y) = d_{\mathrm{prod}}(x,y) + d_{\mathcal V}(P(x),P(y)),
\end{equation}
where
\[
d_{\mathrm{prod}}(x,y)
  \;=\; \sum_{j=1}^\infty 2^{-j}\,\min\{|x_j - y_j|,\,1\}
\]
induces the product topology on $\C^{\N}$.

The basepoint of $\Conf$ is $\widetilde{\N}=(1,2,3,\dots)\in \Conf$. Set
\[
H^{lf}(\infty):=\pi_1^{\mathrm{top}}(\Conf,\widetilde{\N}).
\]

\begin{lemma}[Uniform metric on loop spaces]\label{lem:loop-uniform}
Let $(Z,d)$ be a metric space and $z_0\in Z$. Define the bounded metric $\bar d(u,v)=\min\{d(u,v),1\}$ on $Z$ and the
uniform metric on $\Omega(Z,z_0)$ by
\[
\bar d_\infty(\alpha,\beta)=\sup_{t\in[0,1]} \bar d(\alpha(t),\beta(t)).
\]
Then $\bar d_\infty$ induces the compact--open topology on $\Omega(Z,z_0)$.
\end{lemma}

\begin{proof}
Since $[0,1]$ is compact, the compact--open topology agrees with the topology of uniform convergence with respect to any bounded
compatible metric on $Z$. The metric $\bar d$ is bounded and compatible with $d$, hence $\bar d_\infty$ induces the compact--open
topology on $\Omega(Z,z_0)$.
\end{proof}

\subsection{A uniform finiteness lemma}

\begin{lemma}\label{lem:uniform-compact-finite}
Let $T$ be compact and let $F:T\to \Conf$ be continuous with respect to $d_{\sum}$.
Then for every compact $K\subset\C$ there exists $N=N(K)$ such that
\[
F(t)_m \notin K\qquad\text{for all }t\in T\text{ and all }m>N.
\]
In particular, for each $n$ there exists $N(n)$ such that $F(t)_m\notin K_n$ for all $t\in T$ and all $m>N(n)$.
\end{lemma}

\begin{proof}
Assume not. Then there exists a compact set $K\subset\C$ such that for every $N$ one can find
$t_N\in T$ and an index $m_N>N$ with
\[
x_N:=F(t_N)_{m_N}\in K.
\]
Since $m_N>N$, we have $m_N\to\infty$.
By compactness of $T$ and $K$, after passing to a subsequence we may assume that $t_N\to t_\ast\in T$ and
$x_N\to z\in K$.

Let $A_\ast:=P(F(t_\ast))\subset\C$ be the underlying locally finite subset at $t_\ast$.
A locally finite subset of the Hausdorff space $\C$ is closed and discrete. Hence there exists $r>0$ such that
\begin{equation}\label{eq:isolating-ball}
\overline{B}(z,r)\cap A_\ast\subset\{z\}.
\end{equation}

Choose rational data $q\in\Q(i)$ and $0<r_1<r_2$ in $\Q$ such that
\[
\overline{B}(q,r_2)\subset B(z,r)
\quad\text{and}\quad
z\in B(q,r_1).
\]
Then for all $N\gg0$ we have $x_N\in B(q,r_1)$.

By Lemma~\ref{lem:testfamily-strong}(b), choose an index $j_0$ with
\[
0\le \varphi_{j_0}\le 1,\qquad
\varphi_{j_0}\equiv 1 \text{ on } \overline{B}(q,r_1),\qquad
\supp(\varphi_{j_0})\subset B(q,r_2)\subset B(z,r).
\]
Hence for $N\gg0$, $\varphi_{j_0}(x_N)=1$.

If $\overline{B}(z,r)\cap A_\ast=\varnothing$, then $\sum_{a\in A_\ast}\varphi_{j_0}(a)=0$, while
$\sum_{a\in P(F(t_N))}\varphi_{j_0}(a)\ge 1$ for $N\gg0$ (because $x_N$ contributes $1$).
Thus for $N\gg0$,
\[
\Bigl|\sum_{a\in P(F(t_N))}\varphi_{j_0}(a)-\sum_{a\in A_\ast}\varphi_{j_0}(a)\Bigr|\ge 1.
\]

If $\overline{B}(z,r)\cap A_\ast=\{z\}$, let $i_0$ be the unique index with $F(t_\ast)_{i_0}=z$.
Continuity of $F$ (in particular for the $d_{\mathrm{prod}}$--term) implies $F(t_N)_{i_0}\to z$, hence
$F(t_N)_{i_0}\in B(q,r_1)$ for $N\gg0$.
Also, for $N\gg0$ we have $m_N>i_0$, so $m_N\neq i_0$.
Thus for $N\gg0$ there are at least two distinct points ($x_N$ and $F(t_N)_{i_0}$) of $P(F(t_N))$ in $\overline{B}(q,r_1)$, and therefore
\[
\sum_{a\in P(F(t_N))}\varphi_{j_0}(a)\ge 2
\quad\text{while}\quad
\sum_{a\in A_\ast}\varphi_{j_0}(a)=\varphi_{j_0}(z)=1.
\]
Hence again for $N\gg0$,
\[
\Bigl|\sum_{a\in P(F(t_N))}\varphi_{j_0}(a)-\sum_{a\in A_\ast}\varphi_{j_0}(a)\Bigr|\ge 1.
\]

In either case, for $N\gg0$ the $j_0$-summand in
$d_{\mathcal V}\bigl(P(F(t_N)),P(F(t_\ast))\bigr)$ is bounded below by
$2^{-j_0}\cdot \frac{1}{2}$, so $d_{\sum}(F(t_N),F(t_\ast))\not\to 0$.
This contradicts continuity of $F$ at $t_\ast$.
\end{proof}

\subsection{Discreteness criteria}
The following is a partial result of the main theorem in \cite{CalcutMcCarthy09}. Here we provide a simpler proof.

\begin{proposition}[Discreteness criterion for $\pi_1^{\mathrm{top}}$]\label{pro:piltop-discrete-criterion}
Let $(X,x_0)$ be a based space which is \emph{locally path connected} and \emph{semilocally simply connected}.
Then $\pi_1^{\mathrm{top}}(X,x_0)$ is discrete.
\end{proposition}

\begin{proof}
Let $q:\Omega(X,x_0)\twoheadrightarrow \pi_1^{\mathrm{top}}(X,x_0)$ be the quotient map. It suffices to show that
$q^{-1}(\{1\})$ is open in $\Omega(X,x_0)$.

Let $\alpha\in q^{-1}(\{1\})$ be a null-homotopic based loop. For each $t\in[0,1]$, since $X$ is semilocally simply
connected at $\alpha(t)$, choose an open neighborhood $U_t\ni \alpha(t)$ such that the inclusion
$\pi_1(U_t,\alpha(t))\to \pi_1(X,\alpha(t))$ is trivial. Using local path connectedness, we may assume each $U_t$
is path connected. (Then for any $u\in U_t$, basepoint change within $U_t$ shows that every loop in $U_t$ based at $u$
is also null-homotopic in $X$.)

By compactness of $[0,1]$, choose a subdivision $0=t_0<\cdots<t_m=1$ and indices $t_i^\ast$ such that
\[
\alpha\bigl([t_{i-1},t_i]\bigr)\subset U_{t_i^\ast}\eqqcolon U_i \qquad (1\le i\le m).
\]
For each interior vertex $t_i$ ($1\le i\le m-1$), set $x_i:=\alpha(t_i)\in U_i\cap U_{i+1}$ and choose an open
path-connected neighborhood $W_i\subset U_i\cap U_{i+1}$ containing $x_i$. Define a neighborhood of $\alpha$ in the
compact--open topology by
\[
\mathcal{N}:=\Bigl(\bigcap_{i=1}^m \{\beta\in\Omega(X,x_0): \beta([t_{i-1},t_i])\subset U_i\}\Bigr)
\ \cap\
\Bigl(\bigcap_{i=1}^{m-1} \{\beta\in\Omega(X,x_0): \beta(t_i)\in W_i\}\Bigr).
\]
We claim $\mathcal{N}\subset q^{-1}(\{1\})$.

Fix $\beta\in\mathcal{N}$. For each $i=1,\dots,m-1$, choose a path $\lambda_i$ in $W_i$ from $\alpha(t_i)$ to $\beta(t_i)$.
Set $\lambda_0$ and $\lambda_m$ to be the constant path at $x_0$. For each $i$, define a path
\[
\widetilde{\beta}_i \ :=\  \lambda_{i-1}\cdot \beta|_{[t_{i-1},t_i]}\cdot \lambda_i^{-1}.
\]
Then $\widetilde{\beta}_i$ is a path from $\alpha(t_{i-1})$ to $\alpha(t_i)$ and its image lies in $U_i$
(because $\lambda_{i-1},\beta|_{[t_{i-1},t_i]},\lambda_i$ all lie in $U_i$ by construction).
Hence the loop $\alpha|_{[t_{i-1},t_i]}\cdot \widetilde{\beta}_i^{-1}$ is a loop in $U_i$ and is therefore null-homotopic in $X$.
Thus $\alpha|_{[t_{i-1},t_i]}$ is homotopic relative to endpoints to $\widetilde{\beta}_i$ in $X$.
Concatenating over $i=1,\dots,m$ shows that $\alpha$ is homotopic relative to endpoints to $\beta$, hence $\beta$ is null-homotopic.
Therefore $q^{-1}(\{1\})$ is open.

Thus $\{1\}$ is open in $\pi_1^{\mathrm{top}}(X,x_0)$, and by translations (Lemma~\ref{lem:qtop}) every singleton is open;
hence $\pi_1^{\mathrm{top}}(X,x_0)$ is discrete.
\end{proof}

\begin{corollary}[Discreteness of $\pi_1^{\mathrm{top}}( Conf_n(\C))$]\label{cor:piltop-discrete-Confn}
For each $n\ge1$, the topological fundamental group
\[
P_n=\pi_1^{\mathrm{top}}( Conf_n(\C),(1,\dots,n))
\]
is discrete.
\end{corollary}

\begin{proof}
The manifold $ Conf_n(\C)$ is locally path connected and semilocally simply connected, so the claim follows from
Proposition~\ref{pro:piltop-discrete-criterion}.
\end{proof}

\subsection{Non-semilocal simple connectivity of $(\Conf,d_{\sum})$ and non-discreteness of $H^{lf}(\infty)$}

Recall the basepoint $\widetilde{\N}=(1,2,3,\dots)\in \Conf$, the sum metric
\[
d_{\sum}=d_{\mathrm{prod}}+d_{\mathcal V}\circ(P\times P),
\]
and the topological fundamental group
\[
H^{lf}(\infty):=\pi_1^{\mathrm{top}}(\Conf,\widetilde{\N}),
\]
where $\pi_1^{\mathrm{top}}$ carries the quotient topology induced by the canonical surjection
\[
q:\Omega(\Conf,\widetilde{\N})\twoheadrightarrow \pi_1(\Conf,\widetilde{\N}),\qquad \alpha\mapsto[\alpha].
\]

\begin{proposition}\label{prop:small-essential-loops2}
For every $\varepsilon>0$ there exists an integer $m\ge 3$ and a based loop
$\gamma^{(m)}:[0,1]\to \Conf$ at $\widetilde{\N}$ such that:
\begin{enumerate}[label=\textup{(\roman*)}]
\item $\gamma^{(m)}([0,1])\subset B_{d_{\sum}}(\widetilde{\N},\varepsilon)$;
\item $[\gamma^{(m)}]\neq 1$ in $\pi_1(\Conf,\widetilde{\N})$.
\end{enumerate}
In particular, $(\Conf,d_{\sum})$ is not semilocally simply connected at $\widetilde{\N}$.
\end{proposition}

\begin{proof}
Fix $\varepsilon>0$.

\smallskip
\noindent\textbf{Step 1: choose parameters $J$ and $m$.}
Choose $J\ge 1$ such that
\[
\sum_{j>J}2^{-j}<\frac{\varepsilon}{2}.
\]
Next choose an integer $m$ so large that
\[
m\ge J+3
\qquad\text{and}\qquad
2^{-m}<\frac{\varepsilon}{2}.
\]
Fix also a radius $0<\rho<\frac14$.

\smallskip
\noindent\textbf{Step 2: define a loop $z_m$ in the punctured plane.}
Let $z_m:[0,1]\to \C\setminus\{1,2,\dots,m-1\}$ be a based loop at $m$ which winds once positively around the puncture $m-1$
and avoids the remaining punctures $\{1,\dots,m-2\}$. Concretely, one may take $z_m$ as the concatenation of the three paths:
\begin{enumerate}[label=\textup{(\alph*)}]
\item the straight segment from $m$ to $m-1+\rho$;
\item the positively oriented circle $t\mapsto (m-1)+\rho e^{2\pi i t}$;
\item the straight segment from $m-1+\rho$ back to $m$.
\end{enumerate}
Then $z_m(t)\in \overline{B}(m-1,\rho)\cup [m-1+\rho,m]$ for all $t$.
In particular, since $\rho<\frac14$, we have $z_m(t)\neq k$ for every integer $k> m$, and by construction
$z_m(t)\notin\{1,\dots,m-1\}$ for all $t$.

\smallskip
\noindent\textbf{Step 3: define the loop $\gamma^{(m)}$ in $\Conf$.}
Define $\gamma^{(m)}:[0,1]\to \C^{\N}$ by
\[
\gamma^{(m)}_k(t)\equiv k\quad(k\neq m),
\qquad
\gamma^{(m)}_m(t):=z_m(t).
\]
Then $\gamma^{(m)}(0)=\gamma^{(m)}(1)=\widetilde{\N}$.
Moreover, for each $t$ the coordinate $\gamma^{(m)}_m(t)=z_m(t)$ avoids all integers
$k\neq m$, so the coordinates of $\gamma^{(m)}(t)$ are pairwise distinct.
Hence $\gamma^{(m)}(t)\in\Conf$ for all $t$, i.e.\ $\gamma^{(m)}$ is a based loop in $\Conf$.

\smallskip
\noindent\textbf{Step 4: $\gamma^{(m)}$ is $\varepsilon$--small in $d_{\sum}$.}
Since only the $m$-th coordinate moves, for all $t$,
\[
d_{\mathrm{prod}}\bigl(\gamma^{(m)}(t),\widetilde{\N}\bigr)
=\sum_{k\ge 1}2^{-k}\min\{|\gamma^{(m)}_k(t)-k|,1\}
=2^{-m}\min\{|z_m(t)-m|,1\}\le 2^{-m}<\frac{\varepsilon}{2}.
\]

For the vague term, fix $1\le j\le J$. By Lemma~\ref{lem:testfamily-strong}(c),
$\supp(\varphi_j)\subset K_j\subset K_J$.
Since $m\ge J+3$ and $z_m(t)\in \overline{B}(m-1,\rho)\cup[m-1+\rho,m]$ with $\rho<\frac14$, we have
\[
|z_m(t)|\ge m-1-\rho > J,
\]
hence $z_m(t)\notin K_J$ for all $t$, so $\varphi_j(z_m(t))=0$.
Also $m\notin K_J$, hence $\varphi_j(m)=0$.
All other coordinates are fixed, so for every $t$ and each $1\le j\le J$,
\[
\sum_{a\in P(\gamma^{(m)}(t))}\varphi_j(a)=\sum_{a\in \N}\varphi_j(a),
\]
and therefore the first $J$ summands in $d_{\mathcal V}\bigl(P(\gamma^{(m)}(t)),P(\widetilde{\N})\bigr)$ vanish.
Thus
\[
d_{\mathcal V}\bigl(P(\gamma^{(m)}(t)),P(\widetilde{\N})\bigr)
\le \sum_{j>J}2^{-j}<\frac{\varepsilon}{2}.
\]

Combining the two estimates,
\[
d_{\sum}\bigl(\gamma^{(m)}(t),\widetilde{\N}\bigr)
=d_{\mathrm{prod}}\bigl(\gamma^{(m)}(t),\widetilde{\N}\bigr)
+d_{\mathcal V}\bigl(P(\gamma^{(m)}(t)),\N)\bigr)
<\varepsilon
\qquad( \mbox{ for all }  t\in[0,1]).
\]
Thus $\gamma^{(m)}([0,1])\subset B_{d_{\sum}}(\widetilde{\N},\varepsilon)$, proving \textup{(i)}.

\smallskip
\noindent\textbf{Step 5: $\gamma^{(m)}$ is essential in $\pi_1(\Conf,\widetilde{\N})$.}
Consider the projection $\pr_m:\Conf\to Conf_m(\C)$.
Then $\pr_m\circ\gamma^{(m)}$ is a based loop in $Conf_m(\C)$ at $(1,\dots,m)$ whose first $m-1$ coordinates are constant and
whose $m$-th coordinate is $z_m$.
Let $u_{m-1,m}(t)$ be the associated $S^1$-valued loop
\[
u_{m-1,m}(t)
=\frac{\bigl(\pr_m\circ\gamma^{(m)}\bigr)_m(t)-\bigl(\pr_m\circ\gamma^{(m)}\bigr)_{m-1}(t)}
{\left|\bigl(\pr_m\circ\gamma^{(m)}\bigr)_m(t)-\bigl(\pr_m\circ\gamma^{(m)}\bigr)_{m-1}(t)\right|}
=\frac{z_m(t)-(m-1)}{|z_m(t)-(m-1)|}\in S^1.
\]
By construction, $z_m$ winds once positively around $m-1$, hence $\deg(u_{m-1,m})=1$.
Therefore $[\pr_m\circ\gamma^{(m)}]\neq 1$ in $P_m$.
Since $(\pr_m)_*:\pi_1(\Conf,\widetilde{\N})\to P_m$ is a group homomorphism, it follows that
\[
[\gamma^{(m)}]\neq 1\in \pi_1(\Conf,\widetilde{\N}),
\]
proving \textup{(ii)}.

\smallskip
\noindent\textbf{Step 6: conclude non-semilocal simple connectivity.}
Let $U$ be any neighborhood of $\widetilde{\N}$ in $\Conf$.
Choose $\varepsilon>0$ with $B_{d_{\sum}}(\widetilde{\N},\varepsilon)\subset U$.
By the previous steps there exists an essential based loop $\gamma^{(m)}\subset B_{d_{\sum}}(\widetilde{\N},\varepsilon)\subset U$.
Thus no neighborhood of $\widetilde{\N}$ has trivial inclusion-induced fundamental group, i.e.\ $\Conf$ is not semilocally simply
connected at $\widetilde{\N}$.
\end{proof}

\begin{theorem}[Non-discreteness of $H^{lf}(\infty)$]\label{thm:H-not-discrete2}
The topological fundamental group $H^{lf}(\infty)$
is not discrete.
\end{theorem}

\begin{proof}
Suppose for contradiction that $H^{lf}(\infty)$ is discrete. Then $\{1\}$ is open in $H^{lf}(\infty)$.
Since $q:\Omega(\Conf,\widetilde{\N})\twoheadrightarrow H^{lf}(\infty)$ is a quotient map, the preimage
$q^{-1}(\{1\})$ is open in $\Omega(\Conf,\widetilde{\N})$.

Let $c_{\widetilde{\N}}$ denote the constant loop at $\widetilde{\N}$. Since $c_{\widetilde{\N}}\in q^{-1}(\{1\})$ and the
compact--open topology has a subbasis consisting of sets $\langle K,U\rangle=\{\alpha:\alpha(K)\subset U\}$, there exist
compact sets $K_1,\dots,K_r\subset[0,1]$ and open neighborhoods $U_1,\dots,U_r\subset\Conf$ of $\widetilde{\N}$ such that
\[
c_{\widetilde{\N}}\in \bigcap_{i=1}^r \langle K_i,U_i\rangle \subset q^{-1}(\{1\}).
\]
Set $U:=\bigcap_{i=1}^r U_i$, an open neighborhood of $\widetilde{\N}$ in $\Conf$.
Then for any based loop $\alpha:[0,1]\to U$ we have $\alpha(K_i)\subset U_i$ for each $i$, hence
$\alpha\in \bigcap_{i=1}^r \langle K_i,U_i\rangle\subset q^{-1}(\{1\})$.
Therefore every based loop in $U$ is null-homotopic in $\Conf$, which means that $\Conf$ is semilocally simply connected at
$\widetilde{\N}$. This contradicts
Proposition~\ref{prop:small-essential-loops2}. Hence $H^{lf}(\infty)$ is not discrete.
\end{proof}

\section{Inverse limits and the canonical map $\Theta$}

\begin{definition}\label{def:invlimit}
For $m\geq n$, write $p_{m,n}:P_m\to P_n$ for the standard ``forgetful strand'' homomorphisms
induced by the map $\pr_{m,n}: Conf_m(\C)\rightarrow Conf_n(\C)$, and define the
inverse limit
\[
\varprojlim\nolimits_n P_n
:=\Bigl\{(g_n)_{n\ge 1}\in \prod_{n\ge 1}P_n \ \Big|\ p_{m,n}(g_m)=g_n\ \text{ for all }m\ge n\Bigr\}.
\]
We denote by $\pr_n: \Conf \rightarrow Conf_n(\C)$ the projection map
\[
\pr_n\bigl((a_j)_{j\ge1}\bigr)=(a_1,\dots,a_n)
\]
and the induced group homomorphism by
$$p_n: H^{lf}(\infty) \rightarrow P_n$$
\end{definition}

\begin{theorem}\label{thm:invlimit-complete}
Equip $\prod_{n\ge 1}P_n$ with the metric
\[
d_{\prod}\bigl((g_n),(h_n)\bigr)=\sum_{n=1}^\infty 2^{-n}\,\delta(g_n,h_n),
\qquad
\delta(u,v)=\begin{cases}
0,&u=v,\\
1,&u\ne v.
\end{cases}
\]
Then $d_{\prod}$ is complete and induces the product topology. Moreover, $\varprojlim_n P_n$ is a closed subgroup of
$\prod_{n\ge 1}P_n$, hence $\varprojlim_n P_n$ is a completely metrizable topological group.
\end{theorem}

\begin{proof}
Each $P_n$ is discrete, hence complete for the discrete metric $\delta$. The metric $d_{\prod}$ is the standard weighted
product metric, so it induces the product topology and is complete.

The defining equalities $p_{m,n}(g_m)=g_n$ cut out a closed subset because each $p_{m,n}$ is continuous between discrete spaces.
Thus $\varprojlim_n P_n$ is closed in a complete metric space, hence complete. Being a closed subgroup of a topological group,
it is a topological group.
\end{proof}

\begin{theorem}\label{thm:pn-fibration}
The projection $\pr_n:\Conf\longrightarrow Conf_n(\C)$
is a locally trivial fiber bundle. In particular, $\pr_n$ is a Hurewicz fibration.
\end{theorem}

\begin{proof}
Fix $x=(x_1,\dots,x_n)\in Conf_n(\C)$ and set
\[
\delta:=\min_{i\neq j}|x_i-x_j|>0,
\qquad
r:=\frac{\delta}{2}.
\]
Then the closed disks $\overline{B}(x_i,2r/3)$ are pairwise disjoint. Define an open neighborhood
\[
U_x:=\Bigl\{y=(y_1,\dots,y_n)\in Conf_n(\C):\ |y_i-x_i|<\frac{r}{6}\ \text{ for all }i\Bigr\}.
\]

\medskip
\noindent\textbf{Step 1: A continuous family of compactly supported plane homeomorphisms.}
By \cite[Lemma 4.1]{Teh25} there exists a Lipschitz function $\eta:[0,\infty)\to[0,1]$
(with the properties listed there; in the construction one may arrange $\Lip(\eta)\le 3/r$ as in the proof of
\cite[Lemma 4.1]{Teh25})
such that for any $a\in\C$ and any $p$ with $|p-a|\le r/6$ the map
\[
h_{a,p}(z):=z+(p-a)\eta(|z-a|)
\]
is a homeomorphism, satisfies $h_{a,p}(a)=p$, and equals $\id$ on $\C\setminus B(a,2r/3)$; moreover
$(a,p,z)\mapsto h_{a,p}(z)$ is continuous on $\{|p-a|\le r/6\}$.

For each $y=(y_1,\dots,y_n)\in U_x$, define
\[
H_y:=h_{x_1,y_1}\circ h_{x_2,y_2}\circ\cdots\circ h_{x_n,y_n}.
\]
Because the supports $B(x_i,2r/3)$ are disjoint, $H_y$ is a homeomorphism of $\C$ satisfying
$H_y(x_i)=y_i$ for all $i$, and
\[
H_y=\id\quad\text{on }\C\setminus S,
\qquad
S:=\bigcup_{i=1}^n \overline{B}(x_i,2r/3),
\]
where $S$ is a fixed compact set independent of $y$. Also, $(y,z)\mapsto H_y(z)$ is continuous on $U_x\times\C$,
being a finite composition of the continuous maps from \cite[Lemma 4.1]{Teh25}.

\medskip
\noindent\textbf{Step 2: Continuity of the inverse family $(y,z)\mapsto H_y^{-1}(z)$.}
We first prove a continuity statement for a single bump map.

\smallskip
\noindent\emph{Claim.}
Fix $a\in\C$ and $r>0$ as above. The map
\[
(p,w)\longmapsto h_{a,p}^{-1}(w)
\]
is continuous on $\{(p,w)\in\C^2:\ |p-a|\le r/6\}$.

\smallskip
\noindent\emph{Proof of the claim.}
Write $v:=p-a$ and $g_p(z):=v\,\eta(|z-a|)$ so that $h_{a,p}(z)=z+g_p(z)$.
Using $\Lip(\eta)\le 3/r$ and $|v|\le r/6$, we obtain the uniform bound
\[
\Lip(g_p)\le |v|\Lip(\eta)\le \frac{r}{6}\cdot \frac{3}{r}=\frac12.
\]
Given $w\in\C$, define a contraction
\[
T_{p,w}(z):=w-g_p(z).
\]
Then $\Lip(T_{p,w})\le 1/2$, hence $T_{p,w}$ has a unique fixed point $z=z(p,w)$.
By construction, $z$ satisfies $z=w-g_p(z)$, i.e.\ $h_{a,p}(z)=w$, so $z(p,w)=h_{a,p}^{-1}(w)$.

To prove continuity, fix $(p_0,w_0)$ and let $z_0:=z(p_0,w_0)$. Then
\begin{align*}
|z(p,w)-z_0|
&=|T_{p,w}(z(p,w)) - T_{p_0,w_0}(z_0)|\\
&\le |T_{p,w}(z(p,w)) - T_{p,w}(z_0)|
   + |T_{p,w}(z_0)-T_{p_0,w_0}(z_0)|\\
&\le \frac12|z(p,w)-z_0| + |T_{p,w}(z_0)-T_{p_0,w_0}(z_0)|.
\end{align*}
Hence
\[
|z(p,w)-z_0|
\le 2\,|T_{p,w}(z_0)-T_{p_0,w_0}(z_0)|.
\]
Now $(p,w)\mapsto T_{p,w}(z_0)=w-g_p(z_0)$ is continuous (since $(a,p,z)\mapsto h_{a,p}(z)$ is continuous by
\cite[Lemma 4.1]{Teh25}), so the right-hand side tends to $0$ as $(p,w)\to(p_0,w_0)$.
This proves the claim. \qed

\smallskip
Applying the claim with $a=x_i$, we see that $(y_i,z)\mapsto h_{x_i,y_i}^{-1}(z)$ is continuous.
Therefore
\[
H_y^{-1}=h_{x_n,y_n}^{-1}\circ\cdots\circ h_{x_1,y_1}^{-1},
\]
and $(y,z)\mapsto H_y^{-1}(z)$ is continuous on $U_x\times\C$.
Since each $h_{x_i,y_i}=\id$ on $\C\setminus B(x_i,2r/3)$, the same is true for $h_{x_i,y_i}^{-1}$, hence
\[
H_y^{-1}=\id \quad\text{on }\C\setminus S.
\]

\medskip
\noindent\textbf{Step 3: A local trivialization of $\pr_n$.}
Let the fiber over $x$ be
\[
F_x:=\pr_n^{-1}(x)
=\Bigl\{A=(a_j)_{j\ge1}\in \Conf:\ (a_1,\dots,a_n)=x\Bigr\}.
\]
Define
\[
\Phi:U_x\times F_x\longrightarrow \pr_n^{-1}(U_x),
\qquad
\Phi\bigl(y,(a_j)\bigr):=\bigl(H_y(a_j)\bigr)_{j\ge1}.
\]
This is well-defined: since $H_y$ is injective, the coordinates remain distinct, and since $H_y$ is a homeomorphism,
local finiteness is preserved (compact sets pull back to compact sets).
Moreover
\[
\pr_n\bigl(\Phi(y,(a_j))\bigr)=(H_y(x_1),\dots,H_y(x_n))=(y_1,\dots,y_n)=y.
\]

Define
\[
\Psi:\pr_n^{-1}(U_x)\longrightarrow U_x\times F_x,
\qquad
\Psi\bigl((b_j)\bigr):=\Bigl(\pr_n((b_j)),\ \bigl(H_{\pr_n((b_j))}^{-1}(b_j)\bigr)_{j\ge1}\Bigr).
\]
If $y=\pr_n((b_j))$, then the first $n$ coordinates of $\bigl(H_y^{-1}(b_j)\bigr)_{j\ge1}$ are
\[
\bigl(H_y^{-1}(y_1),\dots,H_y^{-1}(y_n)\bigr)=(x_1,\dots,x_n),
\]
so the second component indeed lies in $F_x$.
Algebraically we have $\Psi\circ\Phi=\id_{U_x\times F_x}$ and $\Phi\circ\Psi=\id_{\pr_n^{-1}(U_x)}$.

\medskip
\noindent\textbf{Step 4: Continuity of $\Phi$ and $\Psi$ for the metric $d_{\sum}$.}
Recall $d_{\sum}=d_{\mathrm{prod}}+d_{\mathcal V}\circ(P\times P)$ and $d_{\mathcal V}$ is a metric
(see \cite[Proposition 2.3]{Teh25}).

\smallskip
\noindent\emph{(i) Product part.}
For each fixed $j$, the map $(y,(a_k))\mapsto H_y(a_j)$ is continuous because $(y,z)\mapsto H_y(z)$ is continuous.
The weighted-series definition of $d_{\mathrm{prod}}$ then implies continuity of $\Phi$ in $d_{\mathrm{prod}}$.
The same argument applies to $\Psi$ using continuity of $(y,z)\mapsto H_y^{-1}(z)$ from Step~2.

\smallskip
\noindent\emph{(ii) Vague part via \cite[Lemma 5.2]{Teh25}.}
We prove continuity of $\Phi$ in the $d_{\mathcal V}$-term; the argument for $\Psi$ is identical with $H_y^{-1}$ in place of $H_y$.

Let $(y_m,A_m)\to(y,A)$ in $U_x\times F_x$.
Set
\[
K:=\{(y,A)\}\cup\{(y_m,A_m):m\ge1\}.
\]
Then $K$ is compact.
Define $F:K\to \Clf$ by $F(y',A'):=P(A')$; this is continuous by definition of $d_{\sum}$.
Define a continuous map
\[
\widetilde\Phi:K\times\C\to\C,
\qquad
\widetilde\Phi\bigl((y',A');z\bigr):=H_{y'}(z).
\]
Each $\widetilde\Phi_{(y',A')}=H_{y'}$ is a homeomorphism, and all are the identity on $\C\setminus S$ with the same fixed compact
$S$ from Step~1. Therefore \cite[Lemma 5.2]{Teh25} applies and yields continuity of
\[
(y',A')\longmapsto (H_{y'})_*\bigl(F(y',A')\bigr)=(H_{y'})_*\bigl(P(A')\bigr)
\]
in the vague topology, hence
\[
d_{\mathcal V}\bigl(P(\Phi(y_m,A_m)),P(\Phi(y,A))\bigr)\to 0.
\]
Thus $\Phi$ is continuous for the $d_{\mathcal V}$-part and hence for $d_{\sum}$.
Applying the same argument to the continuous family $H_{y'}^{-1}$ (Step~2) shows $\Psi$ is also continuous.

\smallskip
Consequently $\Phi$ is a homeomorphism $U_x\times F_x\cong \pr_n^{-1}(U_x)$ commuting with the projections to $U_x$.
Since $x$ was arbitrary, $\pr_n$ is a locally trivial fiber bundle.

\medskip
\noindent\textbf{Step 5: Hurewicz fibration.}
Since $Conf_n(\C)$ is a paracompact manifold, the above locally trivial fiber bundle is numerable and therefore a Hurewicz fibration.
\end{proof}

\section{Locally finite inverse limit}
\begin{definition}\label{def:invlimit-lf-subgroup}
A compatible sequence $(g_n)_{n\geq 1}\in\varprojlim_n P_n$ is called \emph{locally finite} if there exist based loops
\[
\alpha^{(n)}:[0,1]\to  Conf_n(\C)\qquad(n\ge 1)
\]
such that:
\begin{enumerate}[label=\textnormal{(\roman*)}]
\item $\alpha^{(n)}(0)=\alpha^{(n)}(1)=(1,\dots,n)$ and $[\alpha^{(n)}]=g_n\in P_n$;
\item $\pr_{n+1,n}\circ \alpha^{(n+1)}=\alpha^{(n)}$ as maps $[0,1]\to Conf_n(\C)$;
\item for every $\ell\ge 1$ there exists $M(\ell)\ge 1$ such that for all $n\ge M(\ell)$ and all $m$ with
$M(\ell)< m\le n$,
\[
\alpha^{(n)}_m(t)\notin K_{\ell}\qquad\text{for all }t\in[0,1].
\]
\end{enumerate}
Denote the subset of such sequences by
\[
\Bigl(\varprojlim\nolimits_n P_n\Bigr)_{\!\mathrm{lf}}\ \le\ \varprojlim\nolimits_n P_n.
\]
\end{definition}

\begin{lemma}\label{lem:invlimit-lf-subgroup-closed-improved}
The subset \(\bigl(\varprojlim_n P_n\bigr)_{\mathrm{lf}}\) is a subgroup of \(\varprojlim_n P_n\).
\end{lemma}

\begin{proof}
Let \((g_n)_{n\ge1},(h_n)_{n\ge1}\in\bigl(\varprojlim_n P_n\bigr)_{\mathrm{lf}}\).
Choose witnessing based loops \(\alpha^{(n)}\) for \((g_n)\) and \(\beta^{(n)}\) for \((h_n)\) as in
Definition~\ref{def:invlimit-lf-subgroup}.

\medskip
\noindent\textbf{(1) Closure under products.}
For each \(n\ge1\), define the based loop
\[
\gamma^{(n)}:=\alpha^{(n)}*\beta^{(n)}:[0,1]\to Conf_n(\C),
\]
i.e.
\[
\gamma^{(n)}(t)=
\begin{cases}
\alpha^{(n)}(2t), & 0\le t\le \tfrac12,\\[2pt]
\beta^{(n)}(2t-1), & \tfrac12\le t\le 1.
\end{cases}
\]
Then \(\gamma^{(n)}(0)=\gamma^{(n)}(1)=(1,\dots,n)\) and
\[
[\gamma^{(n)}]=[\alpha^{(n)}]\,[\beta^{(n)}]=g_nh_n\in P_n,
\]
so Definition~\ref{def:invlimit-lf-subgroup}\textnormal{(i)} holds for \((g_nh_n)\).

For \textnormal{(ii)}, since \(\pr_{n+1,n}\) is applied pointwise and the above reparametrizations match on both sides,
we have an identity of maps:
\[
\pr_{n+1,n}\circ(\alpha^{(n+1)}*\beta^{(n+1)})
=\bigl(\pr_{n+1,n}\circ\alpha^{(n+1)}\bigr)*\bigl(\pr_{n+1,n}\circ\beta^{(n+1)}\bigr).
\]
Using Definition~\ref{def:invlimit-lf-subgroup}\textnormal{(ii)} for \(\alpha^{(n)}\) and \(\beta^{(n)}\), this becomes
\[
\pr_{n+1,n}\circ\gamma^{(n+1)}=\alpha^{(n)}*\beta^{(n)}=\gamma^{(n)}.
\]
Hence \textnormal{(ii)} holds.

For \textnormal{(iii)}, fix \(\ell\ge1\). Let \(M_\alpha(\ell)\) and \(M_\beta(\ell)\) witness \textnormal{(iii)} for
\(\alpha^{(n)}\) and \(\beta^{(n)}\), respectively, and set
\[
M(\ell):=\max\{M_\alpha(\ell),\,M_\beta(\ell)\}.
\]
Let \(n\ge M(\ell)\) and \(m\) satisfy \(M(\ell)<m\le n\).
Then \(\alpha^{(n)}_m(t)\notin K_\ell\) and \(\beta^{(n)}_m(t)\notin K_\ell\) for all \(t\in[0,1]\).
By the explicit piecewise formula for \(\gamma^{(n)}\),
\[
\gamma^{(n)}_m(t)=
\begin{cases}
\alpha^{(n)}_m(2t), & 0\le t\le \tfrac12,\\
\beta^{(n)}_m(2t-1), & \tfrac12\le t\le 1,
\end{cases}
\]
so \(\gamma^{(n)}_m(t)\notin K_\ell\) for all \(t\in[0,1]\).
Thus \textnormal{(iii)} holds for \(\gamma^{(n)}\), and \((g_nh_n)_{n\ge1}\in(\varprojlim_n P_n)_{\mathrm{lf}}\).

\medskip
\noindent\textbf{(2) Closure under inverses.}
Let \((g_n)_{n\ge1}\in(\varprojlim_n P_n)_{\mathrm{lf}}\) with witnesses \(\alpha^{(n)}\).
Define \(\delta^{(n)}:=\overline{\alpha^{(n)}}\) by time-reversal,
\[
\delta^{(n)}(t):=\alpha^{(n)}(1-t).
\]
Then \(\delta^{(n)}\) is based, \( [\delta^{(n)}]=[\alpha^{(n)}]^{-1}=g_n^{-1}\), so \textnormal{(i)} holds.

For \textnormal{(ii)}, for all \(t\in[0,1]\),
\[
\pr_{n+1,n}\circ\delta^{(n+1)}(t)
=\pr_{n+1,n}\bigl(\alpha^{(n+1)}(1-t)\bigr)
=\alpha^{(n)}(1-t)
=\delta^{(n)}(t),
\]
using \(\pr_{n+1,n}\circ\alpha^{(n+1)}=\alpha^{(n)}\).
For \textnormal{(iii)}, if \(M(\ell)\) witnesses \textnormal{(iii)} for \(\alpha^{(n)}\), then for \(n\ge M(\ell)\) and
\(M(\ell)<m\le n\),
\[
\delta^{(n)}_m(t)=\alpha^{(n)}_m(1-t)\notin K_\ell\qquad( \mbox{ for all }  t\in[0,1]),
\]
so \textnormal{(iii)} holds. Hence \((g_n^{-1})_{n\ge1}\in(\varprojlim_n P_n)_{\mathrm{lf}}\).

With the obvious fact that \((1,1,\dots)\in(\varprojlim_n P_n)_{\mathrm{lf}}\),
\(\bigl(\varprojlim_n P_n\bigr)_{\mathrm{lf}}\) is a subgroup of \(\varprojlim_n P_n\).
\end{proof}

\begin{definition}[Pairwise winding homomorphisms]\label{def:omegaij}
Let $n\ge 2$ and let $\alpha:[0,1]\to Conf_n(\C)$ be a based loop at $(1,\dots,n)$.
For $1\le i<j\le n$ define
\[
u_{ij}^\alpha(t):=\frac{\alpha_j(t)-\alpha_i(t)}{|\alpha_j(t)-\alpha_i(t)|}\in S^1.
\]
Set
\[
\omega_{ij}([\alpha]) := \deg(u_{ij}^\alpha)\in \Z.
\]
\end{definition}

\begin{lemma}\label{lem:omegaij-welldef-hom}
For each $1\le i<j\le n$, the map $\omega_{ij}:P_n\to\Z$ is well-defined and is a group homomorphism.
\end{lemma}

\begin{proof}
If $\alpha\simeq\alpha'$ through based loops in $Conf_n(\C)$ then
$t\mapsto \alpha_j(t)-\alpha_i(t)$ is homotopic in $\C^\times$ to
$t\mapsto \alpha'_j(t)-\alpha'_i(t)$, hence $u_{ij}^\alpha$ is homotopic in $S^1$ to $u_{ij}^{\alpha'}$.
Thus $\deg(u_{ij}^\alpha)=\deg(u_{ij}^{\alpha'})$, so $\omega_{ij}$ is well-defined.

If $\alpha*\beta$ denotes concatenation, then $u_{ij}^{\alpha*\beta}$ is homotopic to the concatenation
$u_{ij}^\alpha*u_{ij}^\beta$ in $S^1$, hence $\deg(u_{ij}^{\alpha*\beta})
=\deg(u_{ij}^\alpha)+\deg(u_{ij}^\beta)$.
Therefore $\omega_{ij}$ is a homomorphism.
\end{proof}

\begin{lemma}\label{lem:disk-avoid-equal}
Fix $\ell\ge 1$.
Let $\alpha:[0,1]\to Conf_n(\C)$ be a based loop such that
\[
\alpha_1([0,1])\cup\alpha_2([0,1])\subset K_\ell
\qquad\text{and}\qquad
\alpha_m([0,1])\cap K_\ell=\varnothing
\]
for some $m\in\{3,\dots,n\}$.
Then $\omega_{1m}([\alpha])=\omega_{2m}([\alpha])$.
\end{lemma}

\begin{proof}
For $s\in[0,1]$ define
\[
H(s,t):=\alpha_m(t)-\bigl((1-s)\alpha_1(t)+s\,\alpha_2(t)\bigr)\in\C.
\]
Since $K_\ell$ is convex and $\alpha_1(t),\alpha_2(t)\in K_\ell$, we have
$(1-s)\alpha_1(t)+s\alpha_2(t)\in K_\ell$ and hence $\bigl|(1-s)\alpha_1(t)+s\alpha_2(t)\bigr|\le \ell$.
On the other hand $\alpha_m(t)\notin K_\ell$ implies $|\alpha_m(t)|>\ell$.
Thus for all $(s,t)\in[0,1]^2$,
\[
|H(s,t)|\ge |\alpha_m(t)|-\bigl|(1-s)\alpha_1(t)+s\alpha_2(t)\bigr|> \ell-\ell=0,
\]
so $H(s,t)\in\C^\times$.
Therefore
\[
\frac{H(s,t)}{|H(s,t)|}\in S^1
\]
gives a homotopy from $u_{1m}^\alpha(t)$ (at $s=0$) to $u_{2m}^\alpha(t)$ (at $s=1$),
so the two loops in $S^1$ have the same degree.
\end{proof}

\begin{construction}[A compatible sequence $(g_n)_{n\geq 1}\in\varprojlim_n P_n$]\label{constr:gn}
For each $m\ge 3$, choose a based loop $\beta^{(m)}:[0,1]\to Conf_m(\C)$ at $(1,\dots,m)$ such that
\[
\beta^{(m)}_k(t)\equiv k\quad(1\le k\le m-1)
\]
and the $m$-th coordinate $\beta^{(m)}_m$ is a loop in
$
\C\setminus\{1,2,\dots,m-1\}
$
based at $m$, homotopic (in that punctured plane) to a small positively oriented circle around the puncture $1$.
(For example: move from $m$ to a point $1+\rho$ above the real axis, traverse the circle
$1+\rho e^{2\pi it}$ with $0<\rho<\tfrac12$, then return along the same path; this avoids
$\{1,2,\dots,m-1\}$.)
Let
\[
a_m:=[\beta^{(m)}]\in P_m.
\]

For $n\ge 3$, view each $a_m$ (with $3\le m\le n$) as an element of $P_n$ via the standard inclusion
``add trivial strands'',
and define
\[
g_1:=1\in P_1,\qquad g_2:=1\in P_2,\qquad g_n:=a_3a_4\cdots a_n\in P_n\quad(n\ge 3).
\]
\end{construction}

\begin{lemma}\label{lem:gn-compatible}
The sequence $(g_n)_{n\ge 1}$ lies in the inverse limit $\varprojlim_n P_n$.
\end{lemma}

\begin{proof}
By construction, $a_{n+1}$ becomes trivial after forgetting the $(n+1)$-st strand, i.e.\ $p_{n+1,n}(a_{n+1})=1$,
and for $m\le n$ we have $p_{n+1,n}(a_m)=a_m$ (since $a_m$ involves only the first $m$ strands and the last strand is trivial).
Hence
\[
p_{n+1,n}(g_{n+1})
=p_{n+1,n}(a_3\cdots a_n a_{n+1})
=a_3\cdots a_n
=g_n.
\]
Thus $(g_n)\in\varprojlim_n P_n$.
\end{proof}

\begin{lemma}[Winding data]\label{lem:gn-winding}
For each $m\ge 3$ we have
\[
\omega_{1m}(g_m)=1
\qquad\text{and}\qquad
\omega_{2m}(g_m)=0.
\]
\end{lemma}

\begin{proof}
In $g_m=a_3a_4\cdots a_m$, every factor $a_k$ with $k<m$ fixes the $m$-th coordinate throughout a representing loop
(because it is included by adding trivial strands), hence contributes $0$ to both $\omega_{1m}$ and $\omega_{2m}$.
The last factor $a_m$ has strands $1$ and $2$ fixed at $1$ and $2$, while the $m$-th coordinate winds once around $1$
and winds zero times around $2$ (by the choice of $\beta^{(m)}_m$ in Construction~\ref{constr:gn}).
Thus $\omega_{1m}(a_m)=1$ and $\omega_{2m}(a_m)=0$.
Since $\omega_{ij}$ is a homomorphism (Lemma~\ref{lem:omegaij-welldef-hom}),
the equalities for $g_m$ follow.
\end{proof}

\begin{proposition}\label{prop:invlimit-strict}
The compatible sequence $(g_n)_{n\ge 1}\in\varprojlim_n P_n$ constructed in
Construction~\ref{constr:gn} does \emph{not} belong to the locally finite subgroup
$\bigl(\varprojlim_n P_n\bigr)_{\mathrm{lf}}$.
Consequently,
\[
\Bigl(\varprojlim\nolimits_n P_n\Bigr)_{\!\mathrm{lf}}
\ \subsetneq\
\varprojlim\nolimits_n P_n.
\]
\end{proposition}

\begin{proof}
Assume for contradiction that $(g_n)\in(\varprojlim_n P_n)_{\mathrm{lf}}$.
By Definition~\ref{def:invlimit-lf-subgroup}, there exist based loops
\[
\alpha^{(n)}:[0,1]\to Conf_n(\C)\qquad(n\ge 1)
\]
such that (i)--(iii) hold.

By (ii), for each fixed $k$ the coordinate path $\alpha^{(n)}_k:[0,1]\to\C$ is independent of $n\ge k$.
Define $\gamma_k:[0,1]\to\C$ by $\gamma_k(t):=\alpha^{(n)}_k(t)$ for any $n\ge k$.
In particular, $\gamma_1,\gamma_2$ are continuous and hence have compact images.
Choose $\ell\ge 1$ such that
\[
\gamma_1([0,1])\cup\gamma_2([0,1])\subset K_\ell.
\]
Now apply (iii) with this $\ell$: there is $M(\ell)$ such that for every $m>\max\{M(\ell),2\}$, taking $n=m$ gives
\[
\alpha^{(m)}_m([0,1])\cap K_\ell=\varnothing.
\]
Since $\alpha^{(m)}_1=\gamma_1$ and $\alpha^{(m)}_2=\gamma_2$, Lemma~\ref{lem:disk-avoid-equal} yields
\[
\omega_{1m}(g_m)=\omega_{1m}([\alpha^{(m)}])=\omega_{2m}([\alpha^{(m)}])=\omega_{2m}(g_m)
\qquad\text{for all }m>\max\{M(\ell),2\}.
\]
This contradicts Lemma~\ref{lem:gn-winding}, which gives $\omega_{1m}(g_m)=1$ and $\omega_{2m}(g_m)=0$ for every $m\ge 3$.
Therefore $(g_n)\notin(\varprojlim_n P_n)_{\mathrm{lf}}$.
\end{proof}

\section{Continuity and bijectivity of the canonical map $\Theta$}
\subsection{The continuity and surjectivity of $\Theta$}
\begin{definition}[The canonical map $\Theta$]\label{def:Theta}
Define
\[
\Theta:H^{lf}(\infty)\longrightarrow \prod_{n\ge1}P_n,
\qquad
\Theta(g):=\bigl(p_n(g)\bigr)_{n\ge1}.
\]
\end{definition}

\begin{lemma}[Functoriality is continuous for $\pi_1^{\mathrm{top}}$]\label{lem:pi1top-functorial-cont}
Let $f:(X,x_0)\to(Y,y_0)$ be a based continuous map. Then the induced map
\[
f_*:\pi_1^{\mathrm{top}}(X,x_0)\longrightarrow \pi_1^{\mathrm{top}}(Y,y_0)
\]
is continuous.
\end{lemma}

\begin{proof}
Let $q_X:\Omega(X,x_0)\twoheadrightarrow \pi_1^{\mathrm{top}}(X,x_0)$ and
$q_Y:\Omega(Y,y_0)\twoheadrightarrow \pi_1^{\mathrm{top}}(Y,y_0)$ be the quotient maps.
The induced map on loop spaces $\Omega(f):\Omega(X,x_0)\to\Omega(Y,y_0)$ is continuous for the compact--open topology.
Moreover the standard identity
\[
q_Y\circ \Omega(f)= f_*\circ q_X
\]
holds by definition of $f_*$. Since $q_X$ is a quotient map, this identity implies that $f_*$ is continuous.
\end{proof}

\begin{lemma}\label{lem:Theta-lands-lf}
For every $g\in H^{lf}(\infty)$, the compatible sequence $\Theta(g)\in\varprojlim_n P_n$ is locally finite in the
sense of Definition~\ref{def:invlimit-lf-subgroup}. Hence $\Theta$ restricts to a map
\[
\Theta:H^{lf}(\infty)\to \Bigl(\varprojlim\nolimits_n P_n\Bigr)_{\!\mathrm{lf}}.
\]
\end{lemma}

\begin{proof}
Fix $g\in H^{lf}(\infty)$ and choose a representative based loop
$\gamma:[0,1]\to \Conf$ with $g=[\gamma]$.
For each $n\ge1$ set $\alpha^{(n)}:=\pr_n\circ\gamma$.
Then $\alpha^{(n)}$ is based and $[\alpha^{(n)}]=p_n(g)$, giving (i) of Definition~\ref{def:invlimit-lf-subgroup}.
Also $\pr_{n+1,n}\circ\alpha^{(n+1)}=\alpha^{(n)}$, giving (ii).

For (iii), fix $\ell\ge1$ and apply Lemma~\ref{lem:uniform-compact-finite} to the compact parameter space $T=[0,1]$,
the continuous map $F=\gamma:[0,1]\to\Conf$, and the compact set $K=K_\ell$.
We obtain $N(\ell)$ such that $\gamma_m(t)\notin K_\ell$ for all $t\in[0,1]$ and all $m>N(\ell)$.
Setting $M(\ell):=N(\ell)$, for every $n\ge M(\ell)$ and every $m$ with $M(\ell)<m\le n$ we have
\[
\alpha_m^{(n)}(t)=\gamma_m(t)\notin K_\ell\qquad( \mbox{ for all }  t\in[0,1]),
\]
which is exactly (iii).
\end{proof}

\begin{theorem}\label{thm:Theta-cont-surj}
The canonical map
\[
\Theta:H^{lf}(\infty)\longrightarrow \Bigl(\varprojlim\nolimits_n P_n\Bigr)_{\!\mathrm{lf}}
\]
is continuous and surjective.
\end{theorem}

\begin{proof}
\textbf{Continuity.}
By Lemma~\ref{lem:pi1top-functorial-cont}, each induced map
\[
p_n: H^{lf}(\infty)\to P_n
\]
is continuous. Therefore the product map
$g\mapsto\bigl(p_n(g)\bigr)_{n\ge1}$ is continuous into $\prod_{n\ge1}P_n$ with the product topology.
By Lemma~\ref{lem:Theta-lands-lf}, the image lies in the subspace
$\bigl(\varprojlim_n P_n\bigr)_{\mathrm{lf}}$, hence $\Theta$ is continuous as a map into that subspace.

\medskip
\textbf{Surjectivity.}
Let $(g_n)_{n\ge1}\in \bigl(\varprojlim_n P_n\bigr)_{\mathrm{lf}}$.
Choose witnessing based loops $\alpha^{(n)}:[0,1]\to Conf_n(\C)$ satisfying (i)--(iii) of
Definition~\ref{def:invlimit-lf-subgroup}.
For each $m\ge1$, define
\[
\gamma_m(t):=\alpha^{(n)}_m(t)\quad\text{for any }n\ge m.
\]
This is well-defined: if $n'\ge n\ge m$, then (ii) implies the first $n$ coordinates of $\alpha^{(n')}$ equal $\alpha^{(n)}$,
hence $\alpha^{(n')}_m=\alpha^{(n)}_m$.
Set $\gamma(t):=(\gamma_1(t),\gamma_2(t),\dots)\in\C^{\N}$.

\smallskip
\emph{Claim 1: $\gamma(t)\in\Conf$ for every $t\in[0,1]$.}
For each $n$, we have
\[
(\gamma_1(t),\dots,\gamma_n(t))=\alpha^{(n)}(t)\in Conf_n(\C),
\]
so the coordinates are pairwise distinct. For local finiteness, fix $\ell\ge1$ and let $M(\ell)$ be as in (iii).
If $m>M(\ell)$, take $n=m$ in (iii) to obtain $\gamma_m(t)=\alpha^{(m)}_m(t)\notin K_\ell$ for all $t$.
Hence for each fixed $t$, only the finitely many indices $m\le M(\ell)$ can meet $K_\ell$.
Thus $\{\gamma_m(t)\}_{m\ge1}$ is locally finite, so $\gamma(t)\in\Conf$.

\smallskip
\emph{Claim 2: $\gamma:[0,1]\to\Conf$ is continuous for $d_{\sum}$.}
Continuity for $d_{\mathrm{prod}}$ is immediate since each coordinate $\gamma_m$ is continuous.
For the vague term, fix $j\ge1$. By Lemma~\ref{lem:testfamily-strong}(c), $\supp(\varphi_j)\subset K_j$.
Let $M:=M(j)$ from (iii) with $\ell=j$. Then for all $m>M$ and all $t$, $\gamma_m(t)\notin K_j$, so
$\varphi_j(\gamma_m(t))=0$. Therefore for all $t$,
\[
\sum_{a\in P(\gamma(t))}\varphi_j(a)=\sum_{m=1}^{M}\varphi_j(\gamma_m(t)),
\]
a finite sum of continuous functions of $t$, hence continuous.
Fix $t_0$. The $j$-th summand in $d_{\mathcal V}(P(\gamma(t)),P(\gamma(t_0)))$ is thus continuous in $t$, and it is bounded
by $2^{-j}$. By the Weierstrass $M$-test, the defining series for $d_{\mathcal V}$ converges uniformly in $t$, so
$t\mapsto d_{\mathcal V}(P(\gamma(t)),P(\gamma(t_0)))$ is continuous. Hence $\gamma$ is continuous for
$d_{\sum}=d_{\mathrm{prod}}+d_{\mathcal V}$.

\smallskip
Since each $\alpha^{(n)}$ is based, we have $\gamma(0)=\gamma(1)=\widetilde{\N}$, so $\gamma$ is a based loop in $\Conf$.
Let $g:=[\gamma]\in H^{lf}(\infty)$.

For each $n$ and all $t$ we have $(\gamma_1(t),\dots,\gamma_n(t))=\alpha^{(n)}(t)$, hence $\pr_n\circ\gamma=\alpha^{(n)}$
as maps. Therefore
\[
p_n(g)=[p_n\circ\gamma]=[\alpha^{(n)}]=g_n,
\]
so $\Theta(g)=(g_n)$. This proves surjectivity.
\end{proof}


\subsection{The injectivity of $\Theta$}

\begin{lemma}\label{lem:fiber-null}
Fix $n\ge 2$.  Let $\beta:[0,1]\to Conf_n(\C)$ be a based loop at $(1,\dots,n)$ such that
\[
\beta_k(t)\equiv k\qquad(1\le k\le n-1).
\]
Set $z(t):=\beta_n(t)$.  Then $z$ is a based loop in $\C\setminus\{1,\dots,n-1\}$ at $n$, and if
\[
\iota:\C\setminus\{1,\dots,n-1\}\longrightarrow Conf_n(\C),\qquad \iota(w):=(1,\dots,n-1,w),
\]
then
\[
[\beta]=\iota_*([z])\in P_n=\pi_1(Conf_n(\C),(1,\dots,n)).
\]
In particular, if $[\beta]=1$ in $P_n$, then $z$ is null-homotopic in $\C\setminus\{1,\dots,n-1\}$.
\end{lemma}

\begin{proof}
Let
\[
p: Conf_n(\C)\to Conf_{n-1}(\C),\qquad p(z_1,\dots,z_n)=(z_1,\dots,z_{n-1})
\]
be the forgetful map. By the Fadell--Neuwirth theorem, $p$ is a locally trivial fiber bundle.
The fiber over $(1,\dots,n-1)$ identifies canonically with $\C\setminus\{1,\dots,n-1\}$ via $\iota$.

Since $\beta_k(t)\equiv k$ for $1\le k\le n-1$, we have $(p\circ\beta)(t)=(1,\dots,n-1)$ for all $t$, hence
$\beta$ is a loop in the fiber $p^{-1}(1,\dots,n-1)$. Writing $z(t)=\beta_n(t)$ gives $\beta=\iota\circ z$, so
\[
[\beta]=\iota_*([z])\in \pi_1( Conf_n(\C),(1,\dots,n)).
\]

To deduce the final claim, we use injectivity of $\iota_*$ on $\pi_1$.
Consider the long exact sequence of homotopy groups for the fibration $p$ at the basepoint $(1,\dots,n)$:
\[
\pi_2\bigl(Conf_{n-1}(\C),(1,\dots,n-1)\bigr)\xrightarrow{\ \partial\ }
\pi_1\bigl(\C\setminus\{1,\dots,n-1\},n\bigr)\xrightarrow{\ \iota_*\ }
\pi_1\bigl( Conf_n(\C),(1,\dots,n)\bigr).
\]
Exactness shows $\ker(\iota_*)=\mathrm{im}(\partial)$. It is known that $Conf_{n-1}(\C)$ is aspherical for $n\geq 2$
and hence $\iota_*$ is injective. Consequently, if $[\beta]=\iota_*([z])=1$, then $[z]=1$ in
$\pi_1(\C\setminus\{1,\dots,n-1\},n)$, so $z$ is null-homotopic relative to endpoints.
\end{proof}

\begin{lemma}\label{lem:pointpush-endpoints}
Let $A\subset\C$ be finite and let $b\in \C\setminus A$.
Let $z_0:[0,1]\to\C\setminus A$ be a based loop at $b$, and let
$z:[0,1]^2\to\C\setminus A$ be a based null-homotopy with
\[
z(0,t)=z_0(t),\qquad z(1,t)\equiv b,\qquad z(s,0)=z(s,1)\equiv b.
\]
Then there exist a compact set $S\subset\C$ and a continuous map
\[
(s,t,w)\longmapsto H_{s,t}(w)\quad\text{from }[0,1]^2\times\C\text{ to }\C
\]
such that for each $(s,t)$ the map $H_{s,t}:\C\to\C$ is a homeomorphism, and:
\begin{enumerate}[label=\textnormal{(\roman*)}]
\item $H_{s,t}(a)=a$ for all $a\in A$;
\item $H_{s,t}=\id$ on $\C\setminus S$ (in particular, $\supp(H_{s,t})\subset S$);
\item $H_{0,t}=\id$ for all $t$, and $H_{s,0}=H_{s,1}=\id$ for all $s$;
\item $H_{s,t}(z_0(t))=z(s,t)$ for all $(s,t)\in[0,1]^2$.
\end{enumerate}
\end{lemma}

\begin{proof}
Set $K:=z([0,1]^2)\subset \C\setminus A$. Then $K$ is compact and $K\cap A=\varnothing$, so
\[
d:=\dist(K,A)>0.
\]
Fix $r>0$ so small that $\overline{B}(K,2r/3)\cap A=\varnothing$.

Choose a Lipschitz function $\eta:[0,\infty)\to[0,1]$ as in \cite[Lemma 4.1]{Teh25}
(with parameter $r$), so that for every $a,p\in\C$ with $|p-a|\le r/6$ the map
\[
h_{a,p}(w):=w+(p-a)\eta(|w-a|)
\]
is a homeomorphism of $\C$, satisfies $h_{a,p}(a)=p$, equals $\id$ on $\C\setminus B(a,2r/3)$,
and the map $(a,p,w)\mapsto h_{a,p}(w)$ is continuous on $\{(a,p):|p-a|\le r/6\}\times\C$.

Since $z$ is uniformly continuous on the compact square $[0,1]^2$, choose $N\in\N$ such that
\[
|z(s,t)-z(s',t)|<r/6\quad \text{whenever } |s-s'|\le 1/N,\ \text{for all }t\in[0,1].
\]
For each $(s,t)\in[0,1]^2$ define subdivision points $s_k:=ks/N$ ($k=0,1,\dots,N$) and points
\[
w_k(s,t):=z(s_k,t)\in K.
\]
Then $|w_k(s,t)-w_{k-1}(s,t)|<r/6$ for every $k=1,\dots,N$.

Define
\[
H_{s,t}
:= h_{w_{N-1}(s,t),\,w_N(s,t)}\circ h_{w_{N-2}(s,t),\,w_{N-1}(s,t)}\circ \cdots \circ h_{w_0(s,t),\,w_1(s,t)}.
\]

\smallskip
\noindent\emph{(i)--(ii).}
Each factor $h_{w_{k-1},w_k}$ equals $\id$ outside $B(w_{k-1},2r/3)\subset B(K,2r/3)$, hence fixes every $a\in A$
(because $\overline{B}(K,2r/3)\cap A=\varnothing$). Therefore $H_{s,t}(a)=a$ for all $a\in A$.
Also every factor is $\id$ on $\C\setminus \overline{B}(K,2r/3)$, so $H_{s,t}=\id$ on $\C\setminus S$ where
\[
S:=\overline{B}(K,2r/3),
\]
independent of $(s,t)$.

\smallskip
\noindent\emph{(iii).}
If $s=0$, then $w_k(0,t)=z(0,t)=z_0(t)$ for all $k$, so each factor is $h_{a,a}=\id$ and hence $H_{0,t}=\id$.
If $t\in\{0,1\}$, then $w_k(s,t)=z(s,t)=b$ for all $k$, so again every factor is $\id$ and $H_{s,0}=H_{s,1}=\id$.

\smallskip
\noindent\emph{(iv).}
We have $w_0(s,t)=z(0,t)=z_0(t)$ and $w_N(s,t)=z(s,t)$, and each factor satisfies
$h_{w_{k-1},w_k}(w_{k-1})=w_k$. Hence
\[
H_{s,t}(z_0(t))=H_{s,t}(w_0(s,t))=w_N(s,t)=z(s,t).
\]

\smallskip
\noindent\emph{Continuity in $(s,t,w)$.}
For each $k$, the map $(s,t)\mapsto (w_{k-1}(s,t),w_k(s,t))$ is continuous, and
$(a,p,w)\mapsto h_{a,p}(w)$ is continuous. Therefore each factor
$(s,t,w)\mapsto h_{w_{k-1}(s,t),w_k(s,t)}(w)$ is continuous; a finite composition of such maps is continuous.
Thus $(s,t,w)\mapsto H_{s,t}(w)$ is continuous on $[0,1]^2\times\C$.
\end{proof}

\begin{lemma}\label{lem:apply-homeo-cont}
Let $K$ be a compact metric space and let
\[
\Phi:K\times\C\to\C,\qquad (x,w)\mapsto \Phi_x(w),
\]
be continuous, where each $\Phi_x:\C\to\C$ is a homeomorphism.
Assume there exists a compact set $S\subset\C$ such that $\Phi_x=\id$ on $\C\setminus S$ for all $x\in K$.
Define the induced map
\[
\Phi^\infty:K\times \Conf\to \Conf,\qquad
\Phi^\infty\bigl(x,(a_m)_{m\ge1}\bigr):=\bigl(\Phi_x(a_m)\bigr)_{m\ge1}.
\]
Then $\Phi^\infty$ is well-defined and continuous with respect to the sum metric
$d_{\sum}$. Consequently, if $\gamma:[0,1]\to\Conf$ is continuous, then
\[
(x,t)\longmapsto \Phi^\infty(x,\gamma(t))
\]
is continuous on $K\times[0,1]$.
\end{lemma}

\begin{proof}
\emph{Well-definedness.}
Fix $x\in K$ and $(a_m)\in\Conf$. Since $\Phi_x$ is injective, the entries $(\Phi_x(a_m))$ are pairwise distinct.
Let $A:=P((a_m))\subset\C$ be the underlying locally finite set. For any compact $C\subset\C$, the preimage
$\Phi_x^{-1}(C)$ is compact, hence $A\cap \Phi_x^{-1}(C)$ is finite. Therefore
\[
\Phi_x(A)\cap C=\Phi_x\bigl(A\cap\Phi_x^{-1}(C)\bigr)
\]
is finite, so $\Phi_x(A)$ is locally finite. Thus $\Phi^\infty(x,(a_m))\in\Conf$.

\smallskip
\emph{Continuity.}
Since $K\times\Conf$ is metric, it suffices to prove sequential continuity.
Let $(x_n,A^{(n)})\to(x,A)$ in $K\times\Conf$, where $A^{(n)}=(a_m^{(n)})$ and $A=(a_m)$, with respect to
$d_K+d_{\sum}$.

\smallskip
\noindent\textbf{(1) Product part.}
Because $d_{\sum}\ge d_{\mathrm{prod}}$, we have $d_{\mathrm{prod}}(A^{(n)},A)\to 0$, hence for each fixed $m$,
$a_m^{(n)}\to a_m$ in $\C$. Fix $\varepsilon>0$ and choose $M$ such that $\sum_{m>M}2^{-m}<\varepsilon/3$.
For each $1\le m\le M$, continuity of $(x,w)\mapsto\Phi_x(w)$ gives
\[
\Phi_{x_n}(a_m^{(n)})\longrightarrow \Phi_x(a_m).
\]
Therefore, for all sufficiently large $n$,
\[
\sum_{m=1}^M 2^{-m}\min\{1,|\Phi_{x_n}(a_m^{(n)})-\Phi_x(a_m)|\}<\varepsilon/3,
\]
and hence
\[
d_{\mathrm{prod}}\bigl(\Phi^\infty(x_n,A^{(n)}),\Phi^\infty(x,A)\bigr)
<\varepsilon/3+\sum_{m>M}2^{-m}<2\varepsilon/3.
\]
Thus $d_{\mathrm{prod}}\bigl(\Phi^\infty(x_n,A^{(n)}),\Phi^\infty(x,A)\bigr)\to 0$.

\smallskip
\noindent\textbf{(2) Vague part.}
Set
\[
K_0:=\{(x,A)\}\cup\{(x_n,A^{(n)}):n\ge 1\}\subset K\times\Conf.
\]
Then $K_0$ is compact. Define
\[
F:K_0\to \Clf,\qquad F(x',B):=P(B).
\]
This is continuous because convergence in $d_{\sum}$ implies convergence of the $d_{\mathcal V}$-term, i.e.\
$P(B_n)\to P(B)$ in $(\Clf,d_{\mathcal V})$.

Define also
\[
\widetilde\Phi:K_0\times\C\to\C,\qquad \widetilde\Phi\bigl((x',B);w\bigr):=\Phi_{x'}(w).
\]
This map is continuous, each $\widetilde\Phi_{(x',B)}=\Phi_{x'}$ is a homeomorphism, and by hypothesis
$\Phi_{x'}=\id$ on $\C\setminus S$ for all $(x',B)\in K_0$ with the same compact $S$.

Therefore \cite[Lemma 5.2]{Teh25} applies and yields continuity of
\[
(x',B)\longmapsto (\Phi_{x'})_*\bigl(P(B)\bigr)=\Phi_{x'}\bigl(P(B)\bigr)
\]
in the vague topology, hence in the metric $d_{\mathcal V}$ (cf.\ \cite[Prop.\ 2.3]{Teh25}). Equivalently,
\[
d_{\mathcal V}\Bigl(P\bigl(\Phi^\infty(x_n,A^{(n)})\bigr),\,P\bigl(\Phi^\infty(x,A)\bigr)\Bigr)\to 0.
\]

\smallskip
\noindent\textbf{(3) Conclusion.}
Combining (1) and (2), $\Phi^\infty$ is continuous.

\smallskip
\emph{Final statement.}
The map $(x,t)\mapsto(x,\gamma(t))$ is continuous $K\times[0,1]\to K\times\Conf$, and $\Phi^\infty$ is continuous,
hence $(x,t)\mapsto \Phi^\infty(x,\gamma(t))$ is continuous.
\end{proof}

\begin{lemma}\label{lem:filldisk-injective}
Fix $R>0$.  Let $A_{\mathrm{out}}\subset \C\setminus K_R$ be finite and let
$A_{\mathrm{in}}\subset \mathrm{int}(K_R)$ be finite \emph{and nonempty}.
Set
\[
X:=\C\setminus\bigl(K_R\cup A_{\mathrm{out}}\bigr),
\qquad
Y:=\C\setminus\bigl(A_{\mathrm{in}}\cup A_{\mathrm{out}}\bigr),
\]
and choose a basepoint $b\in X$.
Then the inclusion $i:X\hookrightarrow Y$ induces an injective homomorphism
\[
i_*:\pi_1(X,b)\longrightarrow \pi_1(Y,b).
\]
Equivalently, if a loop in $X$ is null-homotopic in $Y$, then it is null-homotopic in $X$.
\end{lemma}

\begin{proof}
Choose $\eta>0$ such that
\[
\{z\in\C: R<|z|<R+\eta\}\cap A_{\mathrm{out}}=\varnothing
\qquad\text{and}\qquad
\{z\in\C: R\le |z|\le R+\eta\}\cap A_{\mathrm{in}}=\varnothing.
\]
Define open subsets of $Y$ by
\[
U:=\{z:|z|<R+\eta\}\setminus A_{\mathrm{in}},
\qquad
V:=X=\C\setminus(K_R\cup A_{\mathrm{out}}).
\]
Then $Y=U\cup V$, both $U$ and $V$ are path connected, and
\[
W:=U\cap V=\{z: R<|z|<R+\eta\}
\]
is an annulus, hence $\pi_1(W)\cong\Z$.

Choose a basepoint $b_0\in W$ (e.g.\ $b_0=R+\eta/2$). If $b\neq b_0$, fix a path in $V=X$
from $b$ to $b_0$ and use basepoint change; injectivity is unaffected. Thus it suffices to work at $b_0$.

By van Kampen,
\[
\pi_1(Y,b_0)\ \cong\ \pi_1(U,b_0)\ *_{\pi_1(W,b_0)}\ \pi_1(V,b_0).
\]
In a free product with amalgamation, the canonical maps from $\pi_1(U)$ and $\pi_1(V)$ into the amalgamated product
are injective provided the attaching maps $\pi_1(W)\to\pi_1(U)$ and $\pi_1(W)\to\pi_1(V)$ are injective
(e.g.\ by the normal form theorem; see \cite[Prop.\ 1.20]{Hatcher02}).
Therefore it is enough to show injectivity of $\pi_1(W,b_0)\to\pi_1(U,b_0)$ and $\pi_1(W,b_0)\to\pi_1(V,b_0)$.

\smallskip
\noindent\emph{Injectivity into $V$.}
Let $r_0:=R+\eta/2$. Note $V\subset\{|z|>R\}$, so $z\neq 0$ on $V$.
Define a continuous map
\[
\rho:V\to \{ |z|=r_0\},\qquad \rho(z)=r_0\,\frac{z}{|z|}.
\]
Then $\rho|_{\{|z|=r_0\}}=\id$, so the circle $\{|z|=r_0\}\subset W\subset V$ is a retract of $V$.
Hence the induced map $\pi_1(\{|z|=r_0\},b_0)\to\pi_1(V,b_0)$ is injective.
Since $W$ deformation retracts onto $\{|z|=r_0\}$, the map $\pi_1(W,b_0)\to\pi_1(V,b_0)$ is injective.

\smallskip
\noindent\emph{Injectivity into $U$.}
The space $U$ is a disk punctured at the nonempty finite set $A_{\mathrm{in}}$, hence $U$ deformation retracts onto
a wedge of $|A_{\mathrm{in}}|$ circles and in particular
\[
H_1(U;\Z)\ \cong\ \Z^{|A_{\mathrm{in}}|},
\]
which is torsion free.
Let $\ell$ be the standard generator of $\pi_1(W,b_0)\cong\Z$ represented by the loop $\ell(t)=r_0e^{2\pi it}$.
This loop winds once around every puncture in $A_{\mathrm{in}}\subset\{|z|<R\}$.
Under abelianization (the Hurewicz map) $\pi_1(U,b_0)\to H_1(U;\Z)$, the class of $\ell$ maps to
\[
(1,1,\dots,1)\in\Z^{|A_{\mathrm{in}}|},
\]
which is nonzero and has infinite order. Hence the image of the generator of $\pi_1(W)\cong\Z$ has infinite order in
$\pi_1(U)$, so the homomorphism $\pi_1(W,b_0)\to\pi_1(U,b_0)$ is injective.

\smallskip
Therefore $\pi_1(V,b_0)\to\pi_1(Y,b_0)$ is injective, and since $V=X$ this is exactly injectivity of
$i_*:\pi_1(X,b)\to\pi_1(Y,b)$.
\end{proof}


\begin{proposition}\label{prop:Theta-injective}
The canonical homomorphism
\[
\Theta:\ H^{lf}(\infty)\longrightarrow \Bigl(\varprojlim\nolimits_n P_n\Bigr)_{\!\mathrm{lf}}
\]
is injective.
\end{proposition}

\begin{proof}
Let $g\in H^{lf}(\infty)$ satisfy $\Theta(g)=1$. Choose a based loop
$\gamma:[0,1]\to\Conf$ at $\widetilde{\N}$ with $g=[\gamma]$.

\medskip\noindent
\textbf{Step 1: Tail bounds for $\gamma$.}
For each $\ell\in\N$, apply Lemma~\ref{lem:uniform-compact-finite} with $T=[0,1]$, $F=\gamma$, and $K=K_\ell$
to obtain $N_\ell$ such that
\begin{equation}\label{eq:tail-avoid-main-inj-final}
\gamma_m(t)\notin K_\ell
\qquad
\text{for all }t\in[0,1]\text{ and all }m>N_\ell.
\end{equation}

\medskip\noindent
\textbf{Step 2: Inductive straightening of initial strands.}
We construct based loops $\gamma^{(n)}:[0,1]\to\Conf$ ($n\ge0$) such that:
\begin{enumerate}[label=\textnormal{(\roman*)}]
\item $\gamma^{(0)}=\gamma$;
\item $\pr_n(\gamma^{(n)}(t))=(1,\dots,n)$ for all $t$;
\item $[\gamma^{(n)}]=[\gamma]$ in $\pi_1(\Conf,\widetilde{\N})$.
\end{enumerate}

Assume $\gamma^{(n-1)}$ is constructed ($n\ge1$). Set
\[
\beta:=\pr_n\circ\gamma^{(n-1)}:[0,1]\to Conf_n(\C),
\]
a based loop at $(1,\dots,n)$. Since $\Theta(g)=1$, we have
\[
[\beta]=p_n([\gamma^{(n-1)}])=p_n([\gamma])=p_n(g)=1\in P_n.
\]
By the induction hypothesis, $\beta_k(t)\equiv k$ for $1\le k\le n-1$. Let $z_0(t):=\beta_n(t)$,
a loop in $\C\setminus\{1,\dots,n-1\}$ based at $n$.
By Lemma \ref{lem:fiber-null}, $[z_0]=1$ in $\pi_1(\C\setminus\{1,\dots,n-1\},n)$.

Fix a based null-homotopy
\[
z:[0,1]^2\to \C\setminus\{1,\dots,n-1\}
\]
with
\[
z(0,t)=z_0(t),\qquad z(1,t)\equiv n,\qquad z(s,0)=z(s,1)\equiv n.
\]
Apply Lemma~\ref{lem:pointpush-endpoints} with $A=\{1,\dots,n-1\}$ and $b=n$ to obtain a compact set $S_n\subset\C$ and a
continuous family of homeomorphisms $H^{(n)}_{s,t}$ such that:
\[
H^{(n)}_{s,t}(a)=a\ (a\in A),\qquad H^{(n)}_{s,t}=\id \text{ on }\C\setminus S_n,
\]
\[
H^{(n)}_{0,t}=\id,\qquad H^{(n)}_{s,0}=H^{(n)}_{s,1}=\id,\qquad H^{(n)}_{s,t}(z_0(t))=z(s,t).
\]
Define
\[
\mathcal{H}^{(n)}:[0,1]^2\to\Conf,\qquad
\mathcal{H}^{(n)}(s,t):=\bigl(H^{(n)}_{s,t}(\gamma^{(n-1)}_m(t))\bigr)_{m\ge1}.
\]
By Lemma~\ref{lem:apply-homeo-cont}, $\mathcal{H}^{(n)}$ is continuous and is a based homotopy through based loops.
Set $\gamma^{(n)}:=\mathcal{H}^{(n)}(1,\cdot)$. Then $\pr_n(\gamma^{(n)}(t))=(1,\dots,n)$ and
$[\gamma^{(n)}]=[\gamma^{(n-1)}]=[\gamma]$.

\medskip\noindent
\textbf{Step 3: Eventual support control away from $K_{j+1}$.}
Fix $j\ge1$ and set $R:=j+1+\tfrac12$.
Let $L:=N_{j+2}$ from \eqref{eq:tail-avoid-main-inj-final}. Then for $m>L$,
$\gamma_m([0,1])\cap K_{j+2}=\varnothing$, hence $\gamma_m([0,1])\cap K_R=\varnothing$.
Let $S_{\le L}:=\bigcup_{n=1}^{L} S_n$ (compact). Apply Lemma~\ref{lem:uniform-compact-finite} to $\gamma$
and the compact set $S_{\le L}\cup K_{j+2}$ to obtain $M_R$ such that for $m>M_R$,
\begin{equation}\label{eq:tail-avoid-support-final}
\gamma_m([0,1])\cap(S_{\le L}\cup K_{j+2})=\varnothing.
\end{equation}
In particular, for $m>M_R$, $\gamma_m([0,1])$ avoids each $S_k$ for $1\le k\le L$, so by induction over stages
$1,\dots,L$ we get $\gamma^{(L)}_m=\gamma_m$ and therefore
\[
\gamma^{(L)}_m([0,1])\subset \C\setminus K_R\qquad(m>M_R).
\]

We claim that there exists $n_0(j)$ such that for all $n\ge n_0(j)$ we can choose the stage-$n$ homotopy so that
\begin{equation}\label{eq:support-avoid-KR-final}
S_n\cap K_R=\varnothing,
\qquad\text{hence}\qquad
H^{(n)}_{s,t}\big|_{K_{j+1}}=\id_{K_{j+1}}\ \text{ for all }(s,t).
\end{equation}
Set $n_{\mathrm{start}}:=\max\{L,M_R,j+2\}+1$ and argue by induction on $n\ge n_{\mathrm{start}}$, assuming
$S_k\cap K_R=\varnothing$ for all $L<k<n$. Then each stage $k$ with $L<k<n$ fixes $K_R$ pointwise
(because it is $\id$ on $\C\setminus S_k$ and $K_R\cap S_k=\varnothing$), so it preserves $\C\setminus K_R$.
Since $n>M_R$ and $\gamma^{(L)}_n([0,1])\subset\C\setminus K_R$, it follows that
$\gamma^{(n-1)}_n([0,1])\subset \C\setminus K_R$.

Define
\[
A_{\mathrm{in}}:=\{1,\dots,j+1\}\subset \operatorname{int}(K_R),\qquad
A_{\mathrm{out}}:=\{j+2,\dots,n-1\}\subset \C\setminus K_R,
\]
and set
\[
X:=\C\setminus\bigl(K_R\cup A_{\mathrm{out}}\bigr),
\qquad
Y:=\C\setminus\bigl(A_{\mathrm{in}}\cup A_{\mathrm{out}}\bigr)=\C\setminus\{1,\dots,n-1\}.
\]
Then $t\mapsto\gamma^{(n-1)}_n(t)$ is a based loop in $X$ and is null-homotopic in $Y$ (by Step~2).
By Lemma~\ref{lem:filldisk-injective}, the inclusion $X\hookrightarrow Y$ induces an injection on $\pi_1$,
so the loop is null-homotopic in $X$. Choose a based null-homotopy $z:[0,1]^2\to X$ and set $C_n:=z([0,1]^2)$.
Then $C_n$ is compact and disjoint from $K_R$ and from $\{1,\dots,n-1\}$.

Choose
\[
0<r<\min\Bigl\{\tfrac14\,\dist\bigl(C_n,\{1,\dots,n-1\}\bigr),\ \dist(C_n,K_R)\Bigr\}.
\]
Repeat the bump-map construction in Lemma~\ref{lem:pointpush-endpoints} (using \cite[Lemma~4.1]{Teh25})
so that the resulting family $H^{(n)}_{s,t}$ is supported in
\[
S_n:=\overline{B}(C_n,2r/3).
\]
Since $2r/3<\dist(C_n,K_R)$, we have $S_n\cap K_R=\varnothing$, proving \eqref{eq:support-avoid-KR-final}.
This completes the induction and yields the desired $n_0(j)$.

\medskip\noindent
\textbf{Step 4: Infinite concatenation and continuity at $s=1$.}
Let $s_0:=0$ and $s_n:=1-2^{-n}$. Concatenate the based homotopies $\mathcal H^{(n)}$ on strips
$[s_{n-1},s_n]\times[0,1]$ to obtain $\Gamma:[0,1)\times[0,1]\to\Conf$, and set $\Gamma(1,t):=\widetilde{\N}$.
Then $\Gamma$ is continuous on $[0,1)\times[0,1]$, with $\Gamma(0,t)=\gamma(t)$.

Fix $\varepsilon>0$ and choose $J$ such that $\sum_{k>J}2^{-k}<\varepsilon/2$.
By Step~3 for $j=J$, pick $N_0$ so that all stages $n\ge N_0$ fix $K_{J+1}$ pointwise.
Now apply Lemma~\ref{lem:uniform-compact-finite} to the loop $t\mapsto \Gamma(s_{N_0},t)$ and compact $K_{J+1}$
to obtain $M_0$ such that
\[
\Gamma(s_{N_0},t)_m\notin K_{J+1}\qquad(\forall t,\ \forall m>M_0).
\]

Fix $m>M_0$ and $t\in[0,1]$. By hypothesis, at $s=s_{N_0}$ we have
$\Gamma(s_{N_0},t)_m\in \C\setminus K_{J+1}$.
For $s\ge s_{N_0}$, the evolution of $\Gamma(s,t)$ passes only through stages with index $n\ge N_0$.
On each such stage, the underlying plane homeomorphisms restrict to $\id$ on $K_{J+1}$, hence preserve
$\C\setminus K_{J+1}$ setwise.
Therefore, once the $m$-th point lies in $\C\setminus K_{J+1}$ at $s=s_{N_0}$, it remains in $\C\setminus K_{J+1}$
for all later $s$.
Since $m>M_0$ and $t$ were arbitrary, the stated uniform conclusion follows.

Therefore
\[
\Gamma(s,t)_m\notin K_{J+1}
\qquad
\text{for all }(s,t)\in[s_{N_0},1)\times[0,1]\text{ and all }m>M_0.
\]
Choose $N\ge \max\{N_0,M_0\}$. Then for every $(s,t)\in[s_N,1]\times[0,1]$:
\begin{itemize}
\item stages $>N$ fix $\{1,\dots,N\}$ pointwise, so $\Gamma(s,t)_m\equiv m$ for $1\le m\le N$;
\item for $m>N\ge M_0$, we have $\Gamma(s,t)_m\notin K_{J+1}$.
\end{itemize}
Hence for each $1\le k\le J$, since $\supp(\varphi_k)\subset K_k\subset K_{J+1}$,
\[
\sum_{a\in P(\Gamma(s,t))}\varphi_k(a)=\sum_{m=1}^{N}\varphi_k(m)=\sum_{a\in P(\widetilde{\N})}\varphi_k(a),
\]
so the first $J$ summands in $d_{\mathcal V}(P(\Gamma(s,t)),P(\widetilde{\N}))$ vanish. Thus
\[
d_{\mathcal V}\bigl(P(\Gamma(s,t)),P(\widetilde{\N})\bigr)\le \sum_{k>J}2^{-k}<\varepsilon/2.
\]
Also,
\[
d_{\mathrm{prod}}\bigl(\Gamma(s,t),\widetilde{\N}\bigr)\le \sum_{m>N}2^{-m}<2^{-N},
\]
and enlarging $N$ if necessary we may assume $2^{-N}<\varepsilon/2$. Therefore
\[
d_{\sum}\bigl(\Gamma(s,t),\widetilde{\N}\bigr)
=d_{\mathrm{prod}}\bigl(\Gamma(s,t),\widetilde{\N}\bigr)
+d_{\mathcal V}\bigl(P(\Gamma(s,t)),P(\widetilde{\N})\bigr)
<\varepsilon
\]
for all $(s,t)\in[s_N,1]\times[0,1]$, proving continuity at $(1,t_0)$.

Thus $\Gamma$ extends to a based homotopy from $\gamma$ to the constant loop at $\widetilde{\N}$, so $[\gamma]=1$,
hence $g=1$. Therefore $\ker\Theta=\{1\}$ and $\Theta$ is injective.
\end{proof}


\section{A compatible complete left-invariant metric on $H^{lf}(\infty)$}

\subsection{Moving strands}
\begin{notation}
For $n\geq 1$, we denote
\[
V_n:=\ker\bigl( p_n:H^{lf}(\infty)\to P_n\bigr).
\]
\end{notation}

\begin{lemma}\label{lem:Vn-open-normal}
For each $n\ge 1$, $V_n$ is an open normal subgroup of $H^{lf}(\infty)$, $V_{n+1}\subset V_n$ and $\bigcap_{n\geq 1} V_n=\{1\}$.
\end{lemma}

\begin{proof}
Normality is immediate. The inclusion $V_{n+1}\subset V_n$ follows from $p_n=p_{n+1, n}\circ p_{n+1}$ and functoriality.
For openness, $ Conf_n(\C)$ is a connected manifold, hence locally path-connected and semilocally simply connected.
Therefore by Proposition \ref{pro:piltop-discrete-criterion}, $P_n=\pi_1^{\mathrm{top}}( Conf_n(\C),(1,\dots,n))$ is discrete.
Since $p_n$ is continuous and $\{1\}\subset P_n$ is open, its preimage $V_n$ is open. By Proposition \ref{prop:Theta-injective}, $\Theta: H^{lf}(\infty) \rightarrow (\varprojlim_n P_n)_{\!\mathrm{lf}}$ is injective, we have $\bigcap_{n\geq 1} V_n=\{1\}$.
\end{proof}

The following are some results from \cite{Teh25} that we need.

\begin{theorem}[Katětov--Tong insertion theorem]\label{thm:KT}
Let $K$ be a normal space (in particular, compact metric). If $a:K\to\mathbb{R}$ is upper semicontinuous,
$b:K\to\mathbb{R}$ is lower semicontinuous, and $a(x)<b(x)$ for all $x$, then there exists continuous
$c:K\to\mathbb{R}$ with $a(x)<c(x)<b(x)$ for all $x$.
\end{theorem}

The following result is from \cite[Theorem 5.8]{Teh25}.

\begin{theorem}[Parametric isotopy extension for finite sets (Edwards--Kirby type)]\label{thm:EK}
Let $K$ be a compact metric space, $L\subset K$ closed, and $P$ a finite set. Suppose
$\eta:K\times[0,1]\times P\to\C$ is continuous such that each $\eta_{x,s}:P\to\C$ is injective and
$\eta_{x,s}$ is independent of $s$ for $x\in L$. Let $C:=\eta(K\times[0,1]\times P)$ and let $U\subset\C$
be open with $C\Subset U$. Then there is a continuous family of homeomorphisms
$\Phi_{x,s}:\C\to\C$ with:
\begin{enumerate}[label=(\roman*)]
\item $\Phi_{x,0}=\id$ for all $x$;
\item $\Phi_{x,s}=\id$ for all $x\in L$ and all $s$;
\item $\supp(\Phi_{x,s})\subset U$ for all $(x,s)$;
\item $\Phi_{x,s}(\eta(x,0,p))=\eta(x,s,p)$ for all $(x,s,p)$.
\end{enumerate}
\end{theorem}

The following is \cite[Lemma 5.3]{Teh25}.
\begin{lemma}\label{lem:openness}
Let $K$ be a compact topological space and $F: K \rightarrow \Clf$ be continuous.
If $Q\subset \C$ is compact and $F(x_0)\cap Q=\varnothing$,
then there exists an open neighborhood $U$ of $x_0$ in $K$ such that
$F(x)\cap Q=\varnothing$ for all $x\in U$.
\end{lemma}

\subsection{Straightening the first $n$ strands}

\begin{lemma}\label{lem:straighten-first-n}
If $g\in V_n$, then $g$ admits a based loop representative $\gamma:[0,1]\to\Conf$ such that
\[
\gamma_k(t)\equiv k\qquad(1\le k\le n,\  \mbox{ for all }  t\in[0,1]).
\]
\end{lemma}

\begin{proof}
Choose a based loop $\gamma^{(0)}:[0,1]\to\Conf$ at $\widetilde{\N}$ with $[\gamma^{(0)}]=g$.
Set $\beta(t):=\pr_n(\gamma^{(0)}(t))\in  Conf_n(\C)$, a based loop at $(1,\dots,n)$.

Since $g\in V_n$, the loop $\beta$ is null-homotopic in $ Conf_n(\C)$. Choose a based null-homotopy
\[
H:[0,1]\times[0,1]\to  Conf_n(\C)
\]
with $H(0,t)=\beta(t)$, $H(1,t)\equiv(1,\dots,n)$, and $H(s,0)=H(s,1)=(1,\dots,n)$ for all $s$.

Let $K:=[0,1]$ (parameter $t$), $L:=\{0,1\}\subset K$, and $P:=\{1,\dots,n\}$.
Define
\[
\eta:K\times[0,1]\times P\to\C,\qquad \eta(t,s,i):=\bigl(H(s,t)\bigr)_i.
\]
Then each $\eta_{t,s}$ is injective and $\eta_{t,s}$ is independent of $s$ for $t\in L$.
Let $C:=\eta(K\times[0,1]\times P)$, compact, and choose an open disk $U\subset\C$ with $C\Subset U$.
Apply Theorem~\ref{thm:EK} to obtain $\Phi_{t,s}$ supported in $U$ with $\Phi_{t,0}=\id$, $\Phi_{t,s}=\id$ for $t\in L$,
and $\Phi_{t,s}(\eta(t,0,i))=\eta(t,s,i)$.

Since $\eta(t,0,i)=\beta_i(t)=\gamma^{(0)}_i(t)$, we have $\Phi_{t,1}(\gamma^{(0)}_i(t))=i$ for $1\le i\le n$.
Define $\gamma^{(s)}(t):=\Phi_{t,s}^{\infty}(\gamma^{(0)}(t))$ and set $\gamma:=\gamma^{(1)}$.
By \cite[Lemma 6.3]{Teh25}, $(s,t)\mapsto \gamma^{(s)}(t)$ is continuous in $d_{\sum}$ and based
(since $\Phi_{t,s}=\id$ for $t=0,1$). Then $[\gamma]=[\gamma^{(0)}]=g$ and $\gamma_k\equiv k$ for $k\le n$.
\end{proof}

\subsection{A robust single-circle normalization}

\begin{lemma}\label{lem:single-circle-normalize}
Let $K$ be a compact metric space and $F:K\to C^{lf}_\infty(\C)$ be continuous (vague topology).
Fix $R>0$ and $\varepsilon>0$, and let
\[
I:=\{z\in\C:\ R-\varepsilon\le |z|<R\}.
\]
Define
\[
u(x):=\sup\bigl\{|z|:\ z\in F(x)\cap I\bigr\}\qquad(\sup\emptyset:=R-\varepsilon).
\]
Then $u$ is upper semicontinuous and $u(x)<R$ for all $x$.

Let $L\subset K$ be closed and assume
\begin{equation}\label{eq:relative to-gap}
F(x)\cap\{z:\ R-\varepsilon\le |z|\le R+\varepsilon\}=\varnothing\qquad( \mbox{ for all }  x\in L).
\end{equation}
Then there exist:
\begin{enumerate}[label=(\arabic*)]
\item a continuous function $r:K\to(R-\varepsilon,R]$ such that
\begin{equation}\label{eq:r-properties}
u(x)<r(x)<R \ \text{for }x\in K\setminus L,
\qquad
r(x)=R \ \text{for }x\in L,
\end{equation}
and consequently
\begin{equation}\label{eq:circle-empty}
F(x)\cap\{|z|=r(x)\}=\varnothing\qquad( \mbox{ for all }  x\in K);
\end{equation}
\item a continuous family of radial homeomorphisms $\Psi_{x,t}:\C\to\C$ ($x\in K$, $t\in[0,1]$) such that:
\begin{enumerate}[label=(\alph*)]
\item $\Psi_{x,0}=\id$ for all $x$;
\item $\Psi_{x,t}=\id$ on $\{|z|\le R-\varepsilon\}\cup\{|z|\ge R+\varepsilon\}$ for all $(x,t)$;
\item $\Psi_{x,t}=\id$ for all $x\in L$ and all $t$;
\item $\Psi_{x,1}^{-1}(\{|z|=R\})=\{|z|=r(x)\}$ and hence
\[
(\Psi_{x,1})_*(F(x))\cap\{|z|=R\}=\varnothing\qquad( \mbox{ for all }  x\in K).
\]
\end{enumerate}
\end{enumerate}
\end{lemma}

\begin{proof}
\textbf{Step 1: upper semicontinuity of $u$.}
Let $c<R$. If $c\le R-\varepsilon$, then $\{x:\ u(x)<c\}=\varnothing$, hence open.
Assume $R-\varepsilon<c<R$. For $m\ge1$ set
\[
Q_{c,m}:=\{z:\ c\le |z|\le R-1/m\},
\]
a compact subset of $I$ for all $m$ large enough (otherwise $Q_{c,m}=\varnothing$).
We claim
\[
\{x:\ u(x)<c\}=\bigcap_{m\ge1}\{x:\ F(x)\cap Q_{c,m}=\varnothing\}.
\]
Indeed, if $u(x)<c$ then $F(x)\cap I$ contains no point with $|z|\ge c$, hence $F(x)\cap Q_{c,m}=\varnothing$ for all $m$.
Conversely, if $F(x)\cap Q_{c,m}=\varnothing$ for all $m$, then no point of $F(x)\cap I$ has $|z|\ge c$
(otherwise it would lie in $Q_{c,m}$ for some large $m$), so $u(x)<c$.
Each set $\{x:\ F(x)\cap Q_{c,m}=\varnothing\}$ is open by Lemma~\ref{lem:openness}, hence $\{u<c\}$ is open.
Therefore $u$ is upper semicontinuous. Also $u(x)<R$ since all radii in $I$ are $<R$.

\textbf{Step 2: choose $r$ with the relative constraint.}
By Theorem~\ref{thm:KT} applied to $a=u$ and $b\equiv R$, there exists continuous $r_0:K\to\R$ with
\[
u(x)<r_0(x)<R\qquad( \mbox{ for all }  x\in K).
\]
Since $u(x)\ge R-\varepsilon$, we have $r_0(x)>R-\varepsilon$ for all $x$.

Choose continuous $\lambda:K\to[0,1]$ with $\lambda|_L\equiv0$ and $\lambda>0$ on $K\setminus L$.
Define
\[
r(x):=(1-\lambda(x))\,R+\lambda(x)\,r_0(x).
\]
Then $r|_L\equiv R$. For $x\notin L$, $r(x)<R$ and $r(x)\ge r_0(x)>u(x)$, and $r(x)>R-\varepsilon$.
Thus $r:K\to(R-\varepsilon,R]$ and \eqref{eq:r-properties} holds.

\textbf{Step 3: circle emptiness.}
If $x\notin L$ and $z\in F(x)$ satisfies $|z|=r(x)$, then $z\in I$ and hence $|z|\le u(x)$, contradicting $u(x)<r(x)$.
If $x\in L$, then $r(x)=R$ and \eqref{eq:relative to-gap} gives $F(x)\cap\{|z|=R\}=\varnothing$. Hence \eqref{eq:circle-empty}.

\textbf{Step 4: define a monotone radial homeomorphism $\sigma_x$.}
For each $x\in K$, define $\sigma_x:[0,\infty)\to[0,\infty)$ by
\[
\sigma_x(s):=
\begin{cases}
s, & 0\le s\le R-\varepsilon,\\[4pt]
(R-\varepsilon)+\displaystyle\frac{\varepsilon}{r(x)-(R-\varepsilon)}\,(s-(R-\varepsilon)),
& R-\varepsilon\le s\le r(x),\\[10pt]
R+\displaystyle\frac{\varepsilon}{(R+\varepsilon)-r(x)}\,(s-r(x)),
& r(x)\le s\le R+\varepsilon,\\[10pt]
s, & s\ge R+\varepsilon.
\end{cases}
\]
If $x\in L$, then $r(x)=R$ and the affine pieces reduce to the identity, so $\sigma_x=\id$.
For all $x$, the slopes are positive, hence $\sigma_x$ is continuous and strictly increasing with $\sigma_x(0)=0$
and $\sigma_x(s)\to\infty$ as $s\to\infty$. Thus $\sigma_x$ is a homeomorphism and
\[
\sigma_x(r(x))=R,\qquad \sigma_x(s)=s \text{ for } s\le R-\varepsilon \text{ and } s\ge R+\varepsilon.
\]

\textbf{Step 5: isotopy and definition of $\Psi_{x,t}$.}
Define $\sigma_{x,t}(s):=(1-t)s+t\sigma_x(s)$ for $(x,t)\in K\times[0,1]$.
Then each $\sigma_{x,t}$ is strictly increasing and continuous, hence a homeomorphism, with $\sigma_{x,0}=\id$,
and $\sigma_{x,t}(s)=s$ for $s\le R-\varepsilon$ or $s\ge R+\varepsilon$. If $x\in L$ then $\sigma_{x,t}=\id$ for all $t$.
Define $\Psi_{x,t}(0)=0$ and
\[
\Psi_{x,t}(re^{i\theta})=\sigma_{x,t}(r)e^{i\theta}\qquad(r>0).
\]
Then each $\Psi_{x,t}$ is a homeomorphism of $\C$, supported in $\{R-\varepsilon\le |z|\le R+\varepsilon\}$.

\textbf{Step 6: joint continuity.}
It suffices to show $(x,t,r)\mapsto\sigma_{x,t}(r)$ is continuous on $K\times[0,1]\times[0,\infty)$.
Since $\sigma_{x,t}(r)=(1-t)r+t\sigma_x(r)$, it suffices to show $(x,r)\mapsto\sigma_x(r)$ is continuous.

Fix $(x_0,r_0)$. If $r_0\notin[R-\varepsilon,R+\varepsilon]$, then $\sigma_x(r)=r$ locally, hence continuity is immediate.
Assume $r_0\in[R-\varepsilon,R+\varepsilon]$. Since $r:K\to(R-\varepsilon,R]$ is continuous and $K$ is compact,
$r(K)$ is compact in $(R-\varepsilon,R]$, so
\[
\delta_-:=\min_{x\in K}\bigl(r(x)-(R-\varepsilon)\bigr)>0,
\qquad
\delta_+:=\min_{x\in K}\bigl((R+\varepsilon)-r(x)\bigr)\ge \varepsilon>0.
\]
Hence the coefficients in the affine formulas below are uniformly bounded and continuous in $x$.

Write $\sigma_x$ on $[R-\varepsilon,R+\varepsilon]$ by
\[
\sigma_x(r)=A_-(r(x))\,r+B_-(r(x))\quad (r\le r(x)),
\qquad
\sigma_x(r)=A_+(r(x))\,r+B_+(r(x))\quad (r\ge r(x)),
\]
where for $u\in(R-\varepsilon,R]$,
\[
A_-(u)=\frac{\varepsilon}{u-(R-\varepsilon)},\quad
B_-(u)=(R-\varepsilon)-A_-(u)(R-\varepsilon),
\]
\[
A_+(u)=\frac{\varepsilon}{(R+\varepsilon)-u},\quad
B_+(u)=R-A_+(u)u.
\]
These are continuous functions of $u$, and $u=r(x)$ is continuous in $x$, so the coefficients are continuous in $x$.
A standard two-subsequence argument (according to whether $r_n\le r(x_n)$ or $r_n\ge r(x_n)$) shows that if
$(x_n,r_n)\to(x_0,r_0)$ then $\sigma_{x_n}(r_n)\to\sigma_{x_0}(r_0)$ (the two affine pieces agree at $r=r(x)$).
Thus $(x,r)\mapsto\sigma_x(r)$ is continuous, hence so is $(x,t,r)\mapsto\sigma_{x,t}(r)$, and therefore
$(x,t,z)\mapsto\Psi_{x,t}(z)$ is continuous.

\textbf{Step 7: inverse circle statement.}
Since $\sigma_{x,1}(r(x))=\sigma_x(r(x))=R$ and $\sigma_x$ is strictly increasing, $\sigma_{x,1}^{-1}(R)=r(x)$, i.e.
$\Psi_{x,1}^{-1}(\{|z|=R\})=\{|z|=r(x)\}$. Together with \eqref{eq:circle-empty}, this yields
$(\Psi_{x,1})_*(F(x))\cap\{|z|=R\}=\varnothing$ for all $x$.
\end{proof}

\subsection{Avoiding a prescribed compact for elements of $V_n$}

\begin{lemma}\label{lem:avoid-compact-Vn}
Fix $n\ge 1$, $g\in V_n$.
Then there exists a based loop $\widehat\gamma:[0,1]\to\Conf$ representing $g$ such that
\[
\widehat\gamma_k(t)\equiv k\quad(1\le k\le n),
\qquad
\widehat\gamma_m(t)\notin K_{n-1}
\]
for all $m>n$ and $t\in[0,1]$.
\end{lemma}

\begin{proof}
First choose a based loop $\gamma$ representing $g$ with $\gamma_k\equiv k$ for $k\le n$ by Lemma~\ref{lem:straighten-first-n}.
The case $n=1$ is clear.
Assume $n\ge2$ and set
$R:=n-\frac12, \varepsilon:=\frac14$,
so $K_{n-1}\subset \overline{B}(0,R-\varepsilon)$.

Define $F:[0,1]\to C^{lf}_\infty(\C)$ by $F(t):=P(\gamma(t))$, continuous in the vague topology.
Since $F(0)=F(1)=P(\widetilde{\N})=\N$ and $\N$ avoids $\{z:\ R-\varepsilon\le |z|\le R+\varepsilon\}$,
Lemma~\ref{lem:openness} gives $\delta>0$ with
\[
F(t)\cap\{z:\ R-\varepsilon\le |z|\le R+\varepsilon\}=\varnothing
\]
for all $t\in[0,\delta]\cup[1-\delta,1]$. Let $L:=[0,\delta]\cup[1-\delta,1]\subset[0,1]$.

Apply Lemma~\ref{lem:single-circle-normalize} to this $F$ and $L$ to obtain $\Psi_{t,s}$ supported in the annulus,
equal to $\id$ for $t\in L$, and such that
\begin{equation}\label{eq:heartsuit}
(\Psi_{t,1})_*(F(t))\cap\{|z|=R\}=\varnothing\qquad( \mbox{ for all }  t\in[0,1]).
\end{equation}
Define
\[
\widehat\gamma(t):=\Psi_{t,1}^{\infty}\bigl(\gamma(t)\bigr).
\]
By \cite[Lemma 6.3]{Teh25}, $\widehat\gamma$ is a continuous based loop, and
$H(s,t):=\Psi_{t,s}^{\infty}(\gamma(t))$ is a based homotopy relative to endpoints from $\gamma$ to $\widehat\gamma$,
so $[\widehat\gamma]=[\gamma]=g$.

For $k\le n$, $\widehat\gamma_k(t)=k$ because the annulus $\{R-\varepsilon\le |z|\le R+\varepsilon\}$
is $\{n-\tfrac34\le |z|\le n-\tfrac14\}$, which contains no integer radius.
For $m>n$, we have $|\widehat\gamma_m(0)|=m>R$, and by \eqref{eq:heartsuit} the continuous function
$t\mapsto|\widehat\gamma_m(t)|$ never equals $R$, hence cannot cross below $R$.
Thus $|\widehat\gamma_m(t)|>R$ for all $t$, so $\widehat\gamma_m(t)\notin \overline{B}(0,R)\supset K_{n-1}$.
\end{proof}

\subsection{Profinite basis}

\begin{lemma}\label{lem:small-rep}
Let $\varepsilon>0$. Choose $J$ such that $\sum_{j>J}2^{-j}<\varepsilon/4$ and $n>J+1$.
Then every $g\in V_n$ admits a based loop representative $\widehat\gamma$ with
\[
D(\widehat\gamma,c)<\varepsilon,
\]
where $c$ is the constant loop at $\widetilde{\N}$ and
$D(\widehat\gamma,c)\ =\ \sup_{t\in[0,1]} d_{\sum}\bigl(\widehat\gamma(t),\widetilde{\mathbb N}\bigr)$.
\end{lemma}

\begin{proof}
Let $g\in V_n$. By Lemma~\ref{lem:avoid-compact-Vn}, choose $\widehat\gamma$ representing $g$ with
$\widehat\gamma_k\equiv k$ for $k\le n$ and $\widehat\gamma_m(t)\notin K_J$ for all $m>n$ and $t$.

Fix $t\in[0,1]$. For the product term,
\[
d_{\mathrm{prod}}\bigl(\widehat\gamma(t),\widetilde{\N}\bigr)
=\sum_{m>n}2^{-m}\min\{1,|\widehat\gamma_m(t)-m|\}
\le \sum_{m>n}2^{-m}<\varepsilon/4.
\]

For the vague term, fix $j\le J$. Since $\supp(\varphi_j)\subset K_J$ and all strands $m>n$ avoid $K_J$,
we have $\varphi_j(\widehat\gamma_m(t))=0$ for all $m>n$.
Also $\supp(\varphi_j)\subset K_J\subset\overline{B}(0,n-1)$ implies $\varphi_j(m)=0$ for all integers $m\ge n$,
in particular for all $m>n$. Therefore
\[
\langle P(\widehat\gamma(t)),\varphi_j\rangle=\sum_{m=1}^n \varphi_j(m)=\sum_{m\ge1}\varphi_j(m)
=\langle P(\widetilde{\N}),\varphi_j\rangle.
\]
Hence the first $J$ summands in $d_{\mathcal V}$ vanish and
\[
d_{\mathcal V}\bigl(P(\widehat\gamma(t)),P(\widetilde{\N})\bigr)\le \sum_{j>J}2^{-j}<\varepsilon/4.
\]
Thus $d_{\sum}(\widehat\gamma(t),\widetilde{\N})<\varepsilon/2$ for all $t$, so
$D(\widehat\gamma,c)<\varepsilon/2<\varepsilon$.
\end{proof}

\begin{theorem}[PB]\label{thm:PB}
For every open neighborhood $W\ni 1$ in $H^{lf}(\infty)$, there exists $n$ such that $V_n\subset W$.
\end{theorem}

\begin{proof}
Since $q^{-1}(W)$ is an open neighborhood of the constant loop $c$ in $(\Omega,D)$,
there exists $\varepsilon>0$ such that $B_\varepsilon(c)\subset q^{-1}(W)$.
Choose $n$ as in Lemma~\ref{lem:small-rep} (with this $\varepsilon$). Then for each $g\in V_n$ there exists
a representative $\widehat\gamma\in B_\varepsilon(c)$, hence $g=q(\widehat\gamma)\in W$. Therefore $V_n\subset W$.
\end{proof}

\begin{corollary}[Topological group structure of $H^{lf}(\infty)$]\label{cor:H-topological-group}
The quasitopological group $H^{lf}(\infty)=\pi_1^{\mathrm{top}}(\Conf,\widetilde{\N})$ is a topological group; that is,
the multiplication map
\[
m:H^{lf}(\infty)\times H^{lf}(\infty)\longrightarrow H^{lf}(\infty),\qquad (g,h)\longmapsto gh
\]
is jointly continuous.
\end{corollary}

\begin{proof}
By Theorem~\ref{thm:PB}, the family of open subgroups $\{V_n\}_{n\ge1}$ forms a neighborhood basis at $1$.
Hence $m$ is continuous at $(1,1)$: if $W\ni 1$ is open, choose $n$ with $V_n\subset W$, and then
\[
m(V_n\times V_n)\subset V_n\subset W
\]
since $V_n$ is a subgroup.
For arbitrary $(g_0,h_0)$, use that left and right translations are homeomorphisms
(Lemma~\ref{lem:qtop}) to transport continuity from $(1,1)$ to $(g_0,h_0)$.
Therefore $m$ is jointly continuous on $H^{lf}(\infty)\times H^{lf}(\infty)$.
\end{proof}

\subsection{$\Theta$ is a topological group isomorphism}

\begin{corollary}\label{cor:Theta-top-iso}
The canonical homomorphism
\[
\Theta:\ H^{lf}(\infty)\longrightarrow \Bigl(\varprojlim\nolimits_n P_n\Bigr)_{\!\mathrm{lf}}
\]
is a topological group isomorphism.
\end{corollary}

\begin{proof}
\textbf{Step 1: \(\Theta\) is an abstract group isomorphism.}
By Theorem~\ref{thm:Theta-cont-surj}, \(\Theta\) is continuous and surjective. By
Proposition~\ref{prop:Theta-injective}, \(\Theta\) is injective. Hence \(\Theta\) is a bijective homomorphism.

\smallskip
\textbf{Step 2: A neighborhood basis at \(1\) in the target.}
For \(N\ge1\) define the cylinder set
\[
U_N
:=\Bigl\{(g_n)\in\Bigl(\varprojlim\nolimits_n P_n\Bigr)_{\!\mathrm{lf}}:\ g_1=\cdots=g_N=1\Bigr\}.
\]
Since each \(P_n\) is discrete (Corollary~\ref{cor:piltop-discrete-Confn}), the subset
\[
\widetilde U_N:=\{(g_n)\in\prod_{n\ge1}P_n:\ g_1=\cdots=g_N=1\}
\]
is open in \(\prod_{n\ge1}P_n\), hence \(U_N=\widetilde U_N\cap(\varprojlim_n P_n)_{\mathrm{lf}}\) is open in
\((\varprojlim_n P_n)_{\mathrm{lf}}\).
Moreover \(\{U_N\}_{N\ge1}\) is a neighborhood basis at \(1\) in \((\varprojlim_n P_n)_{\mathrm{lf}}\): if
\(O\subset(\varprojlim_n P_n)_{\mathrm{lf}}\) is open and \(1\in O\), write \(O=\widetilde O\cap(\varprojlim_n P_n)_{\mathrm{lf}}\)
with \(\widetilde O\) open in \(\prod_{n\ge1}P_n\) and \(1\in\widetilde O\). Choose a basic product neighborhood
\(\prod_{n\ge1}O_n\subset \widetilde O\) of \(1\), where each \(O_n\subset P_n\) is open, \(1\in O_n\), and \(O_n=P_n\) for all but
finitely many \(n\). Let \(F:=\{n:\ O_n\neq P_n\}\), finite, and set \(N:=\max F\) (or any \(N\) if \(F=\varnothing\)).
Then \(\widetilde U_N\subset\prod_{n\ge1}O_n\subset\widetilde O\), hence \(U_N\subset O\).

\smallskip
\textbf{Step 3: Continuity of \(\Theta^{-1}\) at \(1\) using PB.}
Let \(W\subset H^{lf}(\infty)\) be an open neighborhood of \(1\).
By Theorem~\ref{thm:PB}, there exists \(N\) such that
\[
V_N:=\ker\bigl(p_N :H^{lf}(\infty)\to P_N\bigr)\subset W.
\]
We claim that
\begin{equation}\label{eq:Theta-inv-cylinder-cor-fixed}
\Theta^{-1}(U_N)=V_N.
\end{equation}
Indeed, for \(g\in H^{lf}(\infty)\),
\[
\Theta(g)\in U_N
\iff p_k(g)=1\ \text{for all }k\le N
\iff p_N(g)=1
\iff g\in V_N.
\]

Thus \eqref{eq:Theta-inv-cylinder-cor-fixed} holds and therefore
\[
\Theta^{-1}(U_N)=V_N\subset W,
\]
so \(\Theta^{-1}\) is continuous at \(1\).

Since left translations are homeomorphisms, we see that
\(\Theta^{-1}\) is continuous everywhere. Therefore \(\Theta\) is a topological group isomorphism.
\end{proof}

\subsection{A truncation lemma for the infinite concatenation}

\begin{lemma}\label{lem:prefix-tail}
Let $(\alpha^{(n)})_{n\ge1}$ be based loops in a space $X$ and set $t_n:=1-2^{-n}$.
Define $\Gamma:[0,1)\to X$ by concatenating $\alpha^{(n)}$ on $[t_{n-1},t_n]$ via affine reparametrization,
and assume $\Gamma$ extends continuously to a based loop on $[0,1]$.
Fix $N\ge1$ and let $\Gamma_{\ge N+1}$ be the tail loop obtained by reparametrizing the restriction
$\Gamma|_{[t_N,1]}$ linearly onto $[0,1]$.
Then $\Gamma$ is based-homotopic (relative to endpoints) to the finite concatenation
\[
\alpha^{(1)}*\alpha^{(2)}*\cdots*\alpha^{(N)}*\Gamma_{\ge N+1}.
\]
\end{lemma}

\begin{proof}
On $[0,t_N]$, the loop $\Gamma$ is exactly the finite concatenation $\alpha^{(1)}*\cdots*\alpha^{(N)}$
up to the fixed affine parametrizations on each interval. On $[t_N,1]$, it is the tail.
A based homotopy between the original parametrization and the parametrization that explicitly realizes
\[
[0,1]=[0,t_N]\cup[t_N,1]
\]
as the concatenation of the prefix and the reparametrized tail is given by a standard family of piecewise-linear
homeomorphisms of $[0,1]$ that fixes $0$ and $1$ and moves the breakpoints continuously so that the restriction
to $[t_N,1]$ is linearly rescaled to $[0,1]$.
Composing $\Gamma$ with this family gives the desired based homotopy relative to endpoints.
\end{proof}

\subsection{Cauchy realization}

\begin{theorem}[CR]\label{thm:CR}
Let $V_n$ be the kernel of $p_n: H^{lf}(\infty) \rightarrow P_n$ for $n\geq 1$.
If $(g_n)_{n\ge 1}$ satisfies $g_{n+1}\in g_nV_n$ for all $n$, then there exists $g\in H^{lf}(\infty)$ such that
\[
g\in g_nV_n\qquad\text{for all }n.
\]
\end{theorem}

\begin{proof}
Define $h_n:=g_n^{-1}g_{n+1}\in V_n$.
For each $n\ge 1$, apply Lemma~\ref{lem:avoid-compact-Vn} to $h_n\in V_n$ to obtain a based loop
$\alpha^{(n)}:[0,1]\to\Conf$ representing $h_n$ such that
\begin{equation}\label{eq:alpha-properties}
\alpha^{(n)}_k(t)\equiv k\ (1\le k\le n),
\qquad
\alpha^{(n)}_m(t)\notin K_{n-1}\ ( \mbox{ for all }  m>n,\  \mbox{ for all }  t).
\end{equation}

Let $t_n:=1-2^{-n}$ and concatenate the loops $\alpha^{(n)}$ on $[t_{n-1},t_n]$ to obtain
$\Gamma:[0,1)\to\Conf$ with $\Gamma(0)=\widetilde{\N}$.

\begin{claim}\label{claim:Gamma-extends}
$\Gamma$ extends continuously to $t=1$ by $\Gamma(1)=\widetilde{\N}$.
\end{claim}

\begin{proof}[Proof of Claim~\ref{claim:Gamma-extends}]
Fix $\eta>0$. We show that for $t$ sufficiently close to $1$,
$d_{\sum}(\Gamma(t),\widetilde{\N})<\eta$.

\emph{Product term.}
Choose $M$ such that $\sum_{m>M}2^{-m}<\eta/2$.
For each $m\le M$, the $m$--th coordinate of $\Gamma$ is identically $m$ after completion of stage $n=m$
because $\alpha^{(n)}$ fixes the first $n$ strands. Hence there exists $T<1$ such that
$\Gamma_m(t)=m$ for all $t\in[T,1)$ and all $m\le M$. Thus for $t\in[T,1)$,
\[
d_{\mathrm{prod}}(\Gamma(t),\widetilde{\N})
=\sum_{m>M}2^{-m}\min\{1,|\Gamma_m(t)-m|\}
\le \sum_{m>M}2^{-m}<\eta/2.
\]

\emph{Vague term.}
Choose $J$ such that $\sum_{j>J}2^{-j}<\eta/2$.
Define $n_*:=\max\{J, M\}$.
We claim that for all $t\in[t_{n_*},1)$ and all $1\le j\le J$,
\begin{equation}\label{eq:vague-equality-uniform}
\langle P(\Gamma(t)),\varphi_j\rangle=\langle P(\widetilde{\N}),\varphi_j\rangle=\langle \N,\varphi_j\rangle.
\end{equation}
Indeed, for such $t$ the point $\Gamma(t)$ lies in some stage $n\ge n_*+1$.
Then strands $1,\dots,n$ are exactly the integers $1,\dots,n$.
Also, by \eqref{eq:alpha-properties}, every strand $m>n$ lies outside $K_{n-1}$,
so $\varphi_j(\Gamma_m(t))=0$ for all $m>n$.
Moreover $\supp(\varphi_j)\subset K_j\subset K_J\subset \overline{B}(0,n-1)$ implies $\varphi_j(m)=0$ for all integers $m>n$.
Therefore
\[
\langle P(\Gamma(t)),\varphi_j\rangle
=\sum_{m=1}^n \varphi_j(m)
=\sum_{m\ge1}\varphi_j(m)
=\langle P(\widetilde{\N}),\varphi_j\rangle,
\]
proving \eqref{eq:vague-equality-uniform}. Consequently, for $t\in[t_{n_*},1)$,
the first $J$ summands in $d_{\mathcal V}$ vanish and
\[
d_{\mathcal V}\bigl(P(\Gamma(t)),P(\widetilde{\N})\bigr)\le \sum_{j>J}2^{-j}<\eta/2.
\]

Combining the bounds gives $d_{\sum}(\Gamma(t),\widetilde{\N})<\eta$ for all $t$ sufficiently close to $1$.
Hence $\Gamma(t)\to\widetilde{\N}$ as $t\to1^-$, so $\Gamma$ extends continuously with $\Gamma(1)=\widetilde{\N}$.
\end{proof}

Let $h:=[\Gamma]\in H^{lf}(\infty)$.

\begin{claim}\label{claim:tail-in-Vn}
For each $n\ge 1$, one has $h\in (h_1\cdots h_{n-1})\,V_n$.
\end{claim}

\begin{proof}[Proof of Claim~\ref{claim:tail-in-Vn}]
Let $\Gamma_{\ge n}:[0,1]\to\Conf$ be the tail loop obtained by reparametrizing $\Gamma|_{[t_{n-1},1]}$ linearly onto $[0,1]$.
For every stage $k\ge n$, the loop $\alpha^{(k)}$ fixes the first $k$ strands and hence fixes the first $n$ strands;
therefore $p_n\circ\Gamma_{\ge n}\equiv(1,\dots,n)$ and so $[\Gamma_{\ge n}]\in V_n$.

By Lemma~\ref{lem:prefix-tail} (with $N=n-1$), $\Gamma$ is based-homotopic to
$\alpha^{(1)}*\cdots*\alpha^{(n-1)}*\Gamma_{\ge n}$. Hence
\[
h=[\Gamma]=[\alpha^{(1)}]\cdots[\alpha^{(n-1)}]\cdot[\Gamma_{\ge n}]
=(h_1\cdots h_{n-1})\cdot v_n
\]
for some $v_n\in V_n$, proving the claim.
\end{proof}

Finally set $g:=g_1h$. Since $g_n=g_1(h_1\cdots h_{n-1})$, Claim~\ref{claim:tail-in-Vn} gives
\[
g=g_1h\in g_1(h_1\cdots h_{n-1})V_n=g_nV_n
\]
for every $n$.
\end{proof}

\subsection{A compatible complete left-invariant ultrametric}

\begin{definition}\label{def:ultrametric}
For $g,h\in H^{lf}(\infty)$ define
\[
N(g,h):=\sup\{n\ge 1:\ g^{-1}h\in V_n\}\in\N\cup\{0,\infty\},
\qquad (\sup\emptyset:=0),
\]
and set
\[
d(g,h):=
\begin{cases}
0,& g=h,\\[2pt]
2^{-(N(g,h)+1)},& g\neq h.
\end{cases}
\]
\end{definition}

The following is the main result of this paper.

\begin{theorem}\label{thm:metric-complete}
The function $d$ is a compatible complete left-invariant ultrametric on $H^{lf}(\infty)$.
Moreover, for every $n\ge 1$ one has
\[
B_d(1,2^{-n})=V_n.
\]
\end{theorem}

\begin{proof}
\textbf{(1) Left invariance and ultrametric inequality.}
Left invariance follows from $N(ag,ah)=N(g,h)$.
If $g^{-1}h\in V_n$ and $h^{-1}\ell\in V_n$, then $g^{-1}\ell\in V_n$,
so
\[
N(g,\ell)\ge \min\{N(g,h),N(h,\ell)\},
\]
equivalent to $d(g,\ell)\le \max\{d(g,h),d(h,\ell)\}$.

\textbf{(2) Separation.}
If $g\neq h$, then $g^{-1}h\neq 1$ and thus $g^{-1}h\notin\bigcap_{n\ge1}V_n=\{1\}$. Hence $N(g,h)<\infty$ and $d(g,h)>0$.

\textbf{(3) Balls at the identity.}
For $n\ge 1$ and $g\neq 1$,
\[
d(g,1)<2^{-n}
\iff
2^{-(N(g,1)+1)}<2^{-n}
\iff
N(g,1)\ge n
\iff
g\in V_n,
\]
using $V_{n+1}\subset V_n$. Also $1\in V_n$ and $d(1,1)=0$, so $B_d(1,2^{-n})=V_n$.

\textbf{(4) Compatibility.}
Each $V_n$ is open and normal (Lemma~\ref{lem:Vn-open-normal}), and by Theorem~\ref{thm:PB}
the family $(V_n)$ is a neighborhood basis at $1$ for the given topology. Since the $d$--balls at $1$
are exactly these $V_n$, the $d$--topology coincides with the given topology.

\textbf{(5) Completeness.}
Let $(x_k)$ be $d$--Cauchy. Choose a subsequence $(x_{k_n})$ such that
\[
d(x_{k_{n+1}},x_{k_n})<2^{-(n+1)}\qquad( \mbox{ for all }  n\ge 1).
\]
Then $N(x_{k_n},x_{k_{n+1}})\ge n$, i.e. $x_{k_n}^{-1}x_{k_{n+1}}\in V_n$, so
\[
x_{k_{n+1}}\in x_{k_n}V_n\qquad( \mbox{ for all }  n).
\]
Apply Theorem~\ref{thm:CR}(CR) to $g_n:=x_{k_n}$ to obtain $x\in H^{lf}(\infty)$ with
$x\in x_{k_n}V_n$ for all $n$, hence $d(x,x_{k_n})<2^{-n}$ and $x_{k_n}\to x$.

Finally, since $(x_k)$ is Cauchy, for each $n$ there exists $K$ such that $d(x_k,x_{k_n})<2^{-n}$
for all $k\ge K$ (with $k_n\ge K$). By the ultrametric inequality,
\[
d(x_k,x)\le \max\{d(x_k,x_{k_n}),d(x_{k_n},x)\}<2^{-n}
\]
for all sufficiently large $k$. Hence $x_k\to x$ and $(H^{lf}(\infty),d)$ is complete.
\end{proof}

\section{$B^{lf}(\infty)$ is complete}

Let $G:=\Aut(\N)$ be regarded as a \emph{discrete} group. Consider the Borel construction
\[
\Conf\HoQuot G:=EG\times_G \Conf
\]
and define
\[
B^{lf}(\infty):=\pi_1^{\mathrm{top}}(\Conf\HoQuot G,\,[e_0,\widetilde{\N}]),
\]
where $e_0\in EG$ is a chosen basepoint.

\subsection{An open completely metrizable subgroup}

Recall that there is a locally trivial fiber bundle
\[
\Conf \longrightarrow \Conf\HoQuot G \xrightarrow{\;\pi\;} BG,
\]
By the asphericity result in \cite{Teh25}, we have an exact sequence of abstract groups
\[
1\longrightarrow \pi_1(\Conf,\widetilde{\N})\xrightarrow{\;\iota_*\;}
\pi_1(\Conf\HoQuot G,[e_0,\widetilde{\N}])\xrightarrow{\;\pi_*\;}
\pi_1(BG,[e_0])\cong G\longrightarrow 1.
\]

\medskip
\noindent\textbf{Convention.}
For a based space $(X,x_0)$ we write $\Omega_1(X,x_0)\subset\Omega(X,x_0)$ for the \emph{path component}
of the constant loop $c_{x_0}$.

The proof of the following result is straightforward.
\begin{lemma}\label{lem:OmegaBG-contractible}
Let $G$ be a discrete group, let $BG$ be a CW model of a classifying space, and let $EG\to BG$ be its universal cover.
Then each path component of $\Omega(BG,[e_0])$ is contractible. In particular, $\Omega_1(BG)$ is contractible.
Moreover, $\Omega_1(BG)$ admits a \emph{based contraction} fixing the constant loop.
\end{lemma}

\begin{lemma}\label{lem:pi1topBG-discrete}
If $G$ is discrete and $BG$ is a CW model of a classifying space, then
$\pi_1^{\mathrm{top}}(BG,[e_0])\cong G$ is discrete.
Equivalently, $\Omega_1(BG)$ is open in $\Omega(BG)$.
\end{lemma}

\begin{proof}
A CW complex is locally path connected and semilocally simply connected. Hence the claim follows from
Proposition~\ref{pro:piltop-discrete-criterion}.
\end{proof}

\begin{lemma}\label{lem:kernel-homeo}
Let $p:(E,x_0)\to (B,b_0)$ be a based \emph{Hurewicz fibration} between compactly generated weak Hausdorff spaces, and let
$F:=p^{-1}(b_0)$ with basepoint $x_0\in F$.
Assume that $\Omega_1(B,b_0)\subset\Omega(B,b_0)$ is \emph{open} and admits a \emph{based contraction} fixing the constant loop,
i.e.\ there exists a continuous map
\[
C:\Omega_1(B,b_0)\times[0,1]\longrightarrow \Omega_1(B,b_0)
\]
with $C(\beta,0)=\beta$, $C(\beta,1)=c_{b_0}$ for all $\beta$, and $C(c_{b_0},s)=c_{b_0}$ for all $s$.
Then the induced homomorphism
\[
\iota_*:\pi_1^{\mathrm{top}}(F,x_0)\longrightarrow \pi_1^{\mathrm{top}}(E,x_0)
\]
is a homeomorphism onto the open subgroup
$\ker\bigl(p_*:\pi_1^{\mathrm{top}}(E,x_0)\to\pi_1^{\mathrm{top}}(B,b_0)\bigr)$.
\end{lemma}

\begin{proof}
Let $q_E:\Omega(E,x_0)\twoheadrightarrow \pi_1^{\mathrm{top}}(E,x_0)$ and
$q_B:\Omega(B,b_0)\twoheadrightarrow \pi_1^{\mathrm{top}}(B,b_0)$ be the quotient maps, and write
$\Omega p:\Omega(E,x_0)\to\Omega(B,b_0)$ for $\alpha\mapsto p\circ\alpha$.

\smallskip
\noindent\textbf{Step 1: an open saturated preimage and openness of the kernel.}
Set
\[
A:=(\Omega p)^{-1}\bigl(\Omega_1(B,b_0)\bigr)\subset \Omega(E,x_0).
\]
Since $\Omega_1(B,b_0)$ is open, $A$ is open.
Moreover $\Omega_1(B,b_0)=q_B^{-1}(\{1\})$, hence
\[
A=q_E^{-1}\Bigl(\ker\bigl(p_*:\pi_1^{\mathrm{top}}(E,x_0)\to\pi_1^{\mathrm{top}}(B,b_0)\bigr)\Bigr).
\]
Because $\Omega_1(B,b_0)$ is open, $\{1\}$ is open in $\pi_1^{\mathrm{top}}(B,b_0)$, so the kernel subgroup is open in
$\pi_1^{\mathrm{top}}(E,x_0)$.

We also note that $A$ is \emph{saturated} for $q_E$: if $\alpha\in A$ and $\alpha\simeq\alpha'$ as based loops in $E$,
then composing the based homotopy with $p$ gives a path in $\Omega(B,b_0)$ from $p\circ\alpha$ to $p\circ\alpha'$.
Since $\Omega_1(B,b_0)$ is the \emph{path component} of $c_{b_0}$ and contains $p\circ\alpha$, it contains the entire image
of that path; hence $p\circ\alpha'\in\Omega_1(B,b_0)$ and $\alpha'\in A$.

\smallskip
\noindent\textbf{Step 2: construct $r:A\to\Omega(F,x_0)$.}
Since $p$ is a Hurewicz fibration and $[0,1]$ is compact Hausdorff, the induced map
$\Omega p:\Omega(E,x_0)\to\Omega(B,b_0)$ is again a Hurewicz fibration (mapping-space preservation in CGWH).
Define
\[
H:A\times[0,1]\longrightarrow \Omega_1(B,b_0),\qquad H(\alpha,s):=C\bigl((\Omega p)(\alpha),s\bigr).
\]
Then $H(\alpha,0)=(\Omega p)(\alpha)$. By the homotopy lifting property for $\Omega p$, there exists a continuous lift
\[
\widetilde H:A\times[0,1]\longrightarrow \Omega(E,x_0)
\]
with $\widetilde H(\alpha,0)=\alpha$ and $(\Omega p)(\widetilde H(\alpha,s))=H(\alpha,s)$.
Define
\[
r:A\longrightarrow \Omega(F,x_0),\qquad r(\alpha):=\widetilde H(\alpha,1).
\]
Indeed, $(\Omega p)(r(\alpha))=H(\alpha,1)=c_{b_0}$, so $r(\alpha)$ is a loop in the fiber $F$.

Let $j:\Omega(F,x_0)\hookrightarrow A$ be inclusion. The homotopy $\widetilde H$ yields a based homotopy in $\Omega(E,x_0)$
from $\id_A$ to $j\circ r$. If $\alpha\in\Omega(F,x_0)$ then $(\Omega p)(\alpha)=c_{b_0}$ and
$H(\alpha,s)=C(c_{b_0},s)=c_{b_0}$ for all $s$, so $(\Omega p)(\widetilde H(\alpha,s))=c_{b_0}$ and hence
$\widetilde H(\alpha,s)\in\Omega(F,x_0)$ for all $s$. Therefore $\widetilde H|_{\Omega(F,x_0)\times[0,1]}$ is a based
homotopy in $\Omega(F,x_0)$ from $\id_{\Omega(F,x_0)}$ to $r\circ j$.

\smallskip
\noindent\textbf{Step 3: descend to $\pi_1^{\mathrm{top}}$ and obtain inverse homeomorphisms.}
Since $A$ is open and saturated, the restriction $q_E|_A:A\twoheadrightarrow \ker(p_*)$ is a quotient map.
Moreover, if $\alpha\simeq\alpha'$ in $\Omega(E,x_0)$ with $\alpha,\alpha'\in A$, then by saturation the entire based homotopy
lies in $A$; composing that homotopy with $r$ yields a based homotopy in $\Omega(F,x_0)$ from $r(\alpha)$ to $r(\alpha')$.
Hence $r$ is constant on the equivalence classes defining $q_E|_A$ and therefore induces a well-defined continuous map
\[
r_*:\ker(p_*)\longrightarrow \pi_1^{\mathrm{top}}(F,x_0),
\qquad
r_*(q_E(\alpha)):=q_F(r(\alpha)),
\]
where $q_F:\Omega(F,x_0)\twoheadrightarrow \pi_1^{\mathrm{top}}(F,x_0)$ is the quotient map.

Similarly, inclusion $j$ induces the continuous homomorphism
\[
j_*=\iota_*:\pi_1^{\mathrm{top}}(F,x_0)\longrightarrow \ker(p_*),
\]
where $\ker(p_*)$ has the subspace topology of $\pi_1^{\mathrm{top}}(E,x_0)$.

The based homotopy in $\Omega(E,x_0)$ between $\id_A$ and $j\circ r$ implies that for every $\alpha\in A$,
$q_E(\alpha)=q_E(j(r(\alpha)))$, hence $\iota_*\circ r_*=\id_{\ker(p_*)}$.
The based homotopy in $\Omega(F,x_0)$ between $\id_{\Omega(F,x_0)}$ and $r\circ j$ implies
$r_*\circ \iota_*=\id_{\pi_1^{\mathrm{top}}(F,x_0)}$.
Thus $\iota_*$ is a bijection with continuous inverse $r_*$, i.e.\ a homeomorphism onto $\ker(p_*)$.
\end{proof}

\begin{corollary}\label{cor:open-complete-subgroup}
For $E=\Conf\HoQuot G$ and $B=BG$ where $G=\Aut(\N)$, the subgroup
\[
K:=\ker\bigl(\pi_*:B^{lf}(\infty)\to G\bigr)
\]
is open in $B^{lf}(\infty)$. Moreover, the map
$
\iota_*:H^{lf}(\infty)\longrightarrow B^{lf}(\infty)
$
is a topological group isomorphism onto $K$. In particular, $K$ is completely metrizable.
\end{corollary}

\begin{proof}
Apply Lemma~\ref{lem:kernel-homeo} to the fibration $\Conf\to \Conf\HoQuot G\xrightarrow{\,\pi\,}BG$. By Lemma~\ref{lem:pi1topBG-discrete}, $\Omega_1(BG)$ is open in $\Omega(BG)$. By Lemma~\ref{lem:OmegaBG-contractible}, $\Omega_1(BG)$ admits a based contraction fixing the constant loop. Hence the hypotheses of Lemma~\ref{lem:kernel-homeo} hold and we conclude that
\[
\iota_*:\pi_1^{\mathrm{top}}(\Conf,\widetilde{\N})\longrightarrow \pi_1^{\mathrm{top}}(\Conf\HoQuot G,[e_0,\widetilde{\N}])
\]
is a homeomorphism onto $K=\ker(\pi_*)$, and that $K$ is open. Finally, by Theorem~\ref{thm:metric-complete}, $H^{lf}(\infty)$ is completely metrizable, so is $K$.
\end{proof}

\subsection{A general metrizability lemma}

\begin{theorem}\label{thm:extend-metric}
Let $G$ be a group equipped with a topology such that left translations are homeomorphisms.
Let $H\le G$ be an \emph{open} subgroup, and assume $H$ is completely metrizable by a complete left-invariant metric $d_H$
inducing the subspace topology on $H$.
Define $d_G:G\times G\to\R$ by
\[
d_G(g,h)=
\begin{cases}
\min\{1,\,d_H(g^{-1}h,1)\},& g^{-1}h\in H,\\[4pt]
1,& g^{-1}h\notin H.
\end{cases}
\]
Then:
\begin{enumerate}[label=\textup{(\alph*)}]
\item $d_G$ is a left-invariant complete metric on $G$;
\item the topology induced by $d_G$ agrees with the given topology on $G$.
\end{enumerate}
In particular, $G$ is completely metrizable.
\end{theorem}

\begin{proof}
\noindent\textbf{Step 1: $d_G$ is a metric and is left-invariant.}
Symmetry and definiteness are immediate. Left-invariance follows since $(kg)^{-1}(kh)=g^{-1}h$.

For the triangle inequality, fix $g,h,k\in G$ and set $u=g^{-1}h$, $v=h^{-1}k$, so $g^{-1}k=uv$.
If $u\notin H$ or $v\notin H$, then $d_G(g,h)=1$ or $d_G(h,k)=1$, hence
$d_G(g,k)\le 1\le d_G(g,h)+d_G(h,k)$.
If $u,v\in H$, then $uv\in H$ and by left-invariance of $d_H$,
\[
d_H(uv,1)\le d_H(u,1)+d_H(v,1),
\]
so
\[
d_G(g,k)=\min\{1,d_H(uv,1)\}\le \min\{1,d_H(u,1)\}+\min\{1,d_H(v,1)\}=d_G(g,h)+d_G(h,k),
\]
using $\min\{1,a+b\}\le \min\{1,a\}+\min\{1,b\}$ for $a,b\ge0$.

\smallskip
\noindent\textbf{Step 2: Compatibility with the given topology.}
It suffices to compare neighborhood bases at $1$.
If $0<\varepsilon<1$, then
\[
B_{d_G}(1,\varepsilon)=\{h\in H:\ d_H(h,1)<\varepsilon\},
\]
which is open in $H$ and hence open in $G$ because $H$ is open.
Conversely, if $U\subset G$ is an open neighborhood of $1$, then $U\cap H$ is open in $H$,
so choose $\varepsilon\in(0,1)$ with $\{h\in H:\ d_H(h,1)<\varepsilon\}\subset U\cap H$.
Then $B_{d_G}(1,\varepsilon)\subset U$.

\smallskip
\noindent\textbf{Step 3: Completeness.}
Let $(g_n)$ be $d_G$-Cauchy. Choose $\varepsilon\in(0,1)$ and $N$ such that $d_G(g_m,g_n)<\varepsilon$ for all $m,n\ge N$.
Then $g_m^{-1}g_n\in H$ for all $m,n\ge N$, hence $g_n\in g_NH$ for all $n\ge N$.
Write $g_n=g_Nh_n$ with $h_n\in H$. By left-invariance,
\[
d_G(g_m,g_n)=d_G(h_m,h_n)=\min\{1,d_H(h_m^{-1}h_n,1)\},
\]
so $(h_n)$ is $d_H$-Cauchy in $H$ and hence converges to some $h\in H$ since $d_H$ is complete.
Therefore $g_n=g_Nh_n\to g_Nh$ in $(G,d_G)$.
\end{proof}

\subsection{Completeness and group topology of $B^{lf}(\infty)$}

\begin{lemma}\label{lem:open-normal-top}
Let $G$ be a \emph{quasitopological group} and let
$H\trianglelefteq G$ be an \emph{open} normal subgroup such that $H$, with the subspace topology,
is a \emph{topological group}.
Then $G$ is a topological group.
\end{lemma}

\begin{proof}
Fix $(a,b)\in G\times G$. Since $H$ is open, $aH$ and $bH$ are open, hence
$(aH)\times (bH)$ is an open neighborhood of $(a,b)$ in $G\times G$.
It suffices to show that multiplication $m:G\times G\to G$ is continuous at $(a,b)$ when restricted to
$(aH)\times (bH)$.

Let $L_a:H\to aH$ and $L_b:H\to bH$ denote left translations, which are homeomorphisms in any quasitopological group.
Consider
\[
\mu_{a,b}:H\times H\longrightarrow G,\qquad (x,y)\longmapsto (ax)(by).
\]
Then $m|_{(aH)\times(bH)}=\mu_{a,b}\circ (L_a^{-1}\times L_b^{-1})$.
Hence it suffices to prove $\mu_{a,b}$ is continuous.

Using normality of $H$ we rewrite, for $x,y\in H$,
\[
(ax)(by)=ab\,(b^{-1}xb)\,y.
\]
The conjugation map $c_b:H\to H$, $c_b(x)=b^{-1}xb$, is continuous because
$c_b=L_{b^{-1}}\circ R_b$ is a composition of translations, and translations are continuous in a
quasitopological group. Multiplication $H\times H\to H$ is continuous by hypothesis, and left translation
by $ab$ is a homeomorphism. Therefore $\mu_{a,b}$ is continuous.
\end{proof}

\begin{theorem}\label{thm:B-complete-top-group}
With the topology from Definition~\ref{def:pi1top}, the locally finite braid group
\[
B^{lf}(\infty)=\pi_1^{\mathrm{top}}(\Conf\HoQuot \Aut(\N),\,[e_0,\widetilde{\N}])
\]
is a \emph{topological group}. Moreover it admits a compatible \emph{complete left-invariant metric};
in particular, it is a complete topological group.
\end{theorem}

\begin{proof}
Let $K=\ker(\pi_*)\le B^{lf}(\infty)$ be the open normal subgroup from
Corollary~\ref{cor:open-complete-subgroup}. By Lemma~\ref{lem:qtop}, $B^{lf}(\infty)$ is a quasitopological group.

By Corollary~\ref{cor:open-complete-subgroup}, $\iota_*:H^{lf}(\infty)\to K$ is a group isomorphism and a homeomorphism.
Hence $K$ is a completely metrizable \emph{topological group} by transporting the group structure and a complete left-invariant
metric from $H^{lf}(\infty)$ via $\iota_*$. Therefore Lemma~\ref{lem:open-normal-top} applies to $(B^{lf}(\infty),K)$ and implies that
$B^{lf}(\infty)$ is a topological group.

Choose a compatible complete left-invariant metric $d_K$ on $K$ which is transported from $H^{lf}(\infty)$.
Applying Theorem~\ref{thm:extend-metric} to the open subgroup $K\le B^{lf}(\infty)$ yields a compatible complete left-invariant metric
on $B^{lf}(\infty)$.
\end{proof}

\section{Density of the finitary braid subgroup}

\subsection{Canonical basepoints and the bar construction}

\begin{convention}\label{conv:bar-basepoints}
For every discrete group $G$ we fix the \emph{bar construction} model
\[
BG:=|B_\bullet G|,\qquad EG:=|E_\bullet G|,
\]
where $B_kG=G^k$ and $E_kG=G^{k+1}$ with the usual face/degeneracy maps, and $EG\to BG$ is the quotient by the free left
$G$--action. We write $b_G\in BG$ for the vertex represented by the empty word, and $e_G\in EG$ for the vertex represented
by $(1)\in E_0G$; thus $p(e_G)=b_G$.
If $\varphi:G\to H$ is a homomorphism, the induced simplicial maps give continuous maps $B\varphi:BG\to BH$ and
$E\varphi:EG\to EH$ satisfying
\[
B\varphi(b_G)=b_H,\qquad E\varphi(e_G)=e_H.
\]
All homotopy quotients $EG\times_G X$ below are formed using these fixed $EG$, and are based at $[e_G,x_0]$.
\end{convention}

\begin{definition}
Let $\Sym_f(\N)\le \Aut(\N)$ denote the subgroup of \emph{finitary permutations}, i.e.\ those
$\sigma\in\Aut(\N)$ such that $\sigma(j)=j$ for all but finitely many $j\in\N$.
\end{definition}

\subsection{Finite braid groups via homotopy quotients}

\begin{definition}\label{def:Bn-hq-bar}
For $n\ge 1$ let $\Sigma_n$ act on $Conf_n(\C)$ by permuting coordinates.
Using Convention~\ref{conv:bar-basepoints}, define
\[
B_n^{\mathrm{top}}
\ :=\
\pi_1^{\mathrm{top}}\!\bigl(E\Sigma_n\times_{\Sigma_n} Conf_n(\C),\ [e_{\Sigma_n},(1,\dots,n)]\bigr).
\]
\end{definition}

\begin{remark}\label{rem:Bn-discrete}
The space $E\Sigma_n\times_{\Sigma_n}Conf_n(\C)\to B\Sigma_n$ is a fiber bundle with fiber $Conf_n(\C)$, and $B\Sigma_n$
is a CW complex. Hence $E\Sigma_n\times_{\Sigma_n}Conf_n(\C)$ is locally path
connected and semilocally simply connected. Therefore $\pi_1^{\mathrm{top}}$ on this space is discrete by
Proposition~\ref{pro:piltop-discrete-criterion}. In particular, $B_n^{\mathrm{top}}$ is a discrete topological group.
\end{remark}

\begin{lemma}\label{lem:finite-braid-exact-bar}
For each $n\ge 1$ there is a short exact sequence of abstract groups
\[
1\longrightarrow P_n \longrightarrow B_n^{\mathrm{top}} \longrightarrow \Sigma_n \longrightarrow 1,
\]
where $P_n=\pi_1(Conf_n(\C),(1,\dots,n))$ and the map $B_n^{\mathrm{top}}\to\Sigma_n$ is induced by the fibration
$Conf_n(\C)\to E\Sigma_n\times_{\Sigma_n}Conf_n(\C)\to B\Sigma_n$.
\end{lemma}

\begin{proof}
The Borel construction yields a (numerable) fibration
\[
Conf_n(\C)\ \longrightarrow\ E\Sigma_n\times_{\Sigma_n} Conf_n(\C)\ \longrightarrow\ B\Sigma_n.
\]
Since $B\Sigma_n$ is a $K(\Sigma_n,1)$, we have $\pi_2(B\Sigma_n)=0$ and $\pi_1(B\Sigma_n,b_{\Sigma_n})\cong \Sigma_n$.
The long exact sequence of homotopy groups therefore contains the exact segment
\[
1=\pi_2(B\Sigma_n)\ \longrightarrow\ \pi_1(Conf_n(\C))\ \longrightarrow\ \pi_1(E\Sigma_n\times_{\Sigma_n}Conf_n(\C))\
\longrightarrow\ \pi_1(B\Sigma_n)\ \longrightarrow\ 1,
\]
which is precisely the stated short exact sequence.
\end{proof}

\subsection{The direct limit subgroup $B_\infty$ inside $B^{lf}(\infty)$}

The following result is clear.
\begin{lemma}\label{lem:jmathn-embed-bar}
For $n\ge 1$ define, for $z=(z_1,\dots,z_n)\in Conf_n(\C)$,
\[
R_n(z):=\max\Bigl\{0,\ \max_{1\le i\le n}|z_i|-n\Bigr\}\in[0,\infty),
\]
and set
\[
\jmath_n(z_1,\dots,z_n):=\bigl(z_1,\dots,z_n,\ n+1+R_n(z),\ n+2+R_n(z),\ \dots \bigr)\ \in\ \C^\N.
\]
Then:
\begin{enumerate}[label=\textup{(\alph*)}]
\item $\jmath_n(z)\in \Conf$ for every $z\in Conf_n(\C)$, and $\jmath_n(1,\dots,n)=\widetilde{\N}$;
\item $\jmath_n$ is continuous, injective, and satisfies $p_n\circ \jmath_n=\id_{Conf_n(\C)}$; in particular,
      $\jmath_n$ is a topological embedding;
\item $\jmath_n$ is $\Sigma_n$--equivariant if $\Sigma_n$ acts by permuting the first $n$ coordinates and fixes the tail.
\end{enumerate}
\end{lemma}

The proof of the following result is a direct verification.
\begin{lemma}\label{lem:snm-embed}
For $1\le n\le m$ define $s_{n,m}:Conf_n(\C)\to Conf_m(\C)$ by
\[
s_{n,m}(z_1,\dots,z_n)
:=
\bigl(z_1,\dots,z_n,\ n+1+R_n(z),\ n+2+R_n(z),\ \dots,\ m+R_n(z)\bigr).
\]
Then:
\begin{enumerate}[label=\textup{(\alph*)}]
\item $s_{n,m}$ is continuous, injective, and satisfies $\pr_{m, n}\circ s_{n,m}=\id_{Conf_n(\C)}$, hence $s_{n,m}$ is a topological embedding;
\item $s_{n,m}(1,\dots,n)=(1,\dots,m)$, so $s_{n,m}$ preserves the standard basepoints;
\item $s_{n,m}$ is $\Sigma_n$--equivariant if $\Sigma_n$ permutes the first $n$ coordinates and fixes the last $m-n$;
\item $\jmath_m\circ s_{n,m}=\jmath_n$ as maps $Conf_n(\C)\to\Conf$.
\end{enumerate}
\end{lemma}

\begin{definition}\label{def:Binfty-bar}
Let $\iota_{n,m}:\Sigma_n\hookrightarrow \Sigma_m$ be the standard inclusion (permute $\{1,\dots,n\}$ and fix $n+1,\dots,m$),
and let $\iota_n:\Sigma_n\hookrightarrow \Aut(\N)$ be the inclusion fixing all $j>n$.
Using Convention~\ref{conv:bar-basepoints} and Lemmas~\ref{lem:snm-embed}(c) and \ref{lem:jmathn-embed-bar}(c),
the pairs $(\iota_{n,m},s_{n,m})$ and $(\iota_n,\jmath_n)$ induce \emph{based} maps of homotopy quotients
\[
S_{n,m}:\ E\Sigma_n\times_{\Sigma_n}Conf_n(\C)\ \longrightarrow\ E\Sigma_m\times_{\Sigma_m}Conf_m(\C),
\]
\[
J_n:\ E\Sigma_n\times_{\Sigma_n}Conf_n(\C)\ \longrightarrow\ E\Aut(\N)\times_{\Aut(\N)}\Conf.
\]
Define the induced homomorphisms on topological fundamental groups
\[
(S_{n,m})_*:B_n^{\mathrm{top}}\to B_m^{\mathrm{top}},
\qquad
(J_n)_*:B_n^{\mathrm{top}}\to B^{lf}(\infty):=\pi_1^{\mathrm{top}}\!\bigl(E\Aut(\N)\times_{\Aut(\N)}\Conf,\,
[e_{\Aut(\N)},\widetilde{\N}]\bigr).
\]
We write
\[
B_n := (J_n)_*\bigl(B_n^{\mathrm{top}}\bigr)\ \le\ B^{lf}(\infty),
\qquad
B_\infty := \bigcup_{n\ge1} B_n\ \le\ B^{lf}(\infty).
\]
Similarly, define $P_n^{\mathrm{img}}:=(\jmath_n)_*(P_n)\le H^{lf}(\infty):=\pi_1^{\mathrm{top}}(\Conf,\widetilde{\N})$ and
\[
P_\infty := \bigcup_{n\ge1} P_n^{\mathrm{img}}\ \le\ H^{lf}(\infty).
\]
\end{definition}

\begin{lemma}\label{lem:directed-compatibility-bar}
For $1\le n\le m$ one has $J_n = J_m\circ S_{n,m}$ as based maps. Consequently,
\[
(J_n)_*=(J_m)_*\circ (S_{n,m})_*,
\qquad
B_n\subseteq B_m,
\qquad
P_n^{\mathrm{img}}\subseteq P_m^{\mathrm{img}}.
\]
\end{lemma}

\begin{proof}
The inclusions $\iota_n=\iota_m\circ \iota_{n,m}$ hold in $\Aut(\N)$, and Lemma~\ref{lem:snm-embed}(d) gives
$\jmath_n=\jmath_m\circ s_{n,m}$. By the functoriality of the bar construction and the Borel construction (with canonical
basepoints fixed in Convention~\ref{conv:bar-basepoints}), these identities imply $J_n=J_m\circ S_{n,m}$ as based maps, and
hence the equalities on $\pi_1^{\mathrm{top}}$. The inclusions $B_n\subseteq B_m$ and $P_n^{\mathrm{img}}\subseteq P_m^{\mathrm{img}}$
follow immediately.
\end{proof}

\begin{lemma}\label{lem:Bn-injects-bar}
For each $n\ge1$, the homomorphism $(J_n)_*:B_n^{\mathrm{top}}\to B^{lf}(\infty)$ is injective.
Moreover, the restriction of $\pi_*:B^{lf}(\infty)\to\Aut(\N)$ to $B_n$ coincides with the extension-by-identity
of the usual map $B_n^{\mathrm{top}}\to\Sigma_n$.
\end{lemma}

Hence the notation
$B_n\le B^{lf}(\infty)$ and $B_\infty=\bigcup_n B_n$ is justified.

\begin{proof}
Consider the commutative diagram of fibrations induced by $(\iota_n,\jmath_n)$:
\[
\begin{CD}
Conf_n(\C) @>>> E\Sigma_n\times_{\Sigma_n}Conf_n(\C) @>>> B\Sigma_n\\
@V\jmath_nVV @VVJ_nV @VV B\iota_n V\\
\Conf @>>> E\Aut(\N)\times_{\Aut(\N)}\Conf @>>> B\Aut(\N).
\end{CD}
\]
Naturality of the associated long exact sequences gives commutativity on $\pi_1$.

\smallskip\noindent
\emph{Permutation compatibility.}
Composing the middle vertical map with $\pi_*:\pi_1^{\mathrm{top}}(E\Aut(\N)\times_{\Aut(\N)}\Conf)\to \pi_1(B\Aut(\N))\cong \Aut(\N)$
agrees with $\iota_n$ composed with $B_n^{\mathrm{top}}\to \Sigma_n\cong \pi_1(B\Sigma_n)$, hence restricts to the
extension-by-identity map on $B_n$.

\smallskip\noindent
\emph{Injectivity.}
Let $\beta\in B_n^{\mathrm{top}}$ and suppose $(J_n)_*(\beta)=1$ in $B^{lf}(\infty)$. Applying $\pi_*$ gives
$\iota_n(\overline\beta)=1$ in $\Aut(\N)$, hence $\overline\beta=1$ in $\Sigma_n$, so $\beta$ lies in the kernel of
$B_n^{\mathrm{top}}\to\Sigma_n$. By Lemma~\ref{lem:finite-braid-exact-bar}, $\beta$ is the image of some
$\alpha\in P_n=\pi_1(Conf_n(\C),(1,\dots,n))$ under the fiber inclusion
\[
(i_n)_*:P_n\longrightarrow B_n^{\mathrm{top}}.
\]

Let $i:\Conf\hookrightarrow E\Aut(\N)\times_{\Aut(\N)}\Conf$ denote the fiber inclusion over $b_{\Aut(\N)}$.
Naturality gives
\[
(J_n)_*\bigl((i_n)_*(\alpha)\bigr)= i_*\bigl((\jmath_n)_*(\alpha)\bigr).
\]
Thus $1=(J_n)_*(\beta)=i_*\bigl((\jmath_n)_*(\alpha)\bigr)$.

Since $B\Aut(\N)$ is a $K(\Aut(\N),1)$, we have $\pi_2(B\Aut(\N))=0$, hence the long exact sequence of the fibration
$\Conf\to E\Aut(\N)\times_{\Aut(\N)}\Conf\to B\Aut(\N)$ shows that
\[
i_*:\pi_1(\Conf,\widetilde{\N})\longrightarrow \pi_1(E\Aut(\N)\times_{\Aut(\N)}\Conf,[e_{\Aut(\N)},\widetilde{\N}])
\]
is injective. Therefore $(\jmath_n)_*(\alpha)=1$ in $H^{lf}(\infty)=\pi_1^{\mathrm{top}}(\Conf,\widetilde{\N})$.

Now apply $ p_n:H^{lf}(\infty)\to P_n$. Since $\pr_n\circ \jmath_n=\id_{Conf_n(\C)}$, we have
\[
 p_n\circ (\jmath_n)_* = \id_{P_n},
\]
so $\alpha= p_n((\jmath_n)_*(\alpha))=1$. Hence $\beta=(i_n)_*(\alpha)=1$, proving $(J_n)_*$ injective.
\end{proof}

\begin{lemma}\label{lem:image-Binfty-bar}
Under the surjective homomorphism $\pi_*:B^{lf}(\infty)\to \Aut(\N)$, one has
\[
\pi_*(B_\infty)=\Sym_f(\N).
\]
\end{lemma}

\begin{proof}
If $\beta\in B_n$, then $\pi_*(\beta)$ acts as an element of $\Sigma_n$ on $\{1,\dots,n\}$ and fixes all $j>n$,
hence $\pi_*(\beta)\in \Sym_f(\N)$. Thus $\pi_*(B_\infty)\subseteq \Sym_f(\N)$.

Conversely, let $\sigma\in \Sym_f(\N)$. Choose $n$ such that $\sigma(j)=j$ for all $j>n$; then $\sigma$ is the
extension-by-identity of a unique $\sigma_n\in \Sigma_n$.
By Lemma~\ref{lem:finite-braid-exact-bar}, $B_n^{\mathrm{top}}\to\Sigma_n$ is surjective, so pick $\beta\in B_n^{\mathrm{top}}$
mapping to $\sigma_n$. Then $(J_n)_*(\beta)\in B_n\subseteq B_\infty$ and Lemma~\ref{lem:Bn-injects-bar} yields
$\pi_*((J_n)_*(\beta))=\sigma$.
\end{proof}

\subsection{Density of $P_\infty$ in $H^{lf}(\infty)$}

\begin{theorem}\label{thm:P-infty-dense-bar}
Let $P_\infty\le H^{lf}(\infty)$ be as in Definition~\ref{def:Binfty-bar}. Then $P_\infty$ is dense in $H^{lf}(\infty)$.
\end{theorem}

\begin{proof}
For each $n\ge 1$, by Theorem~\ref{thm:PB}, $\{V_n\}_{n\ge1}$ is a neighborhood basis at $1$ in $H^{lf}(\infty)$ where
$V_n:=\ker\bigl(p_n:H^{lf}(\infty)\to P_n\bigr)$.

Fix $g\in H^{lf}(\infty)$. We claim there exists $b\in P_\infty$ with
$p_n(b)= p_n(g)$, which then implies $ p_n(gb^{-1})=1$, i.e.\ $b\in gV_n$.

Let $s_n:P_n\to H^{lf}(\infty)$ denote the inclusion $P_n\xrightarrow{(\jmath_n)_*}P_n^{\mathrm{img}}\le H^{lf}(\infty)$.
Then $ p_n\circ s_n=\id_{P_n}$ because $p_n\circ (\jmath_n)_*=\id$.
Set
\[
b:=s_n\bigl( p_n(g)\bigr)\in P_n^{\mathrm{img}}\subseteq P_\infty.
\]
Then $ p_n(b)= p_n(g)$ as required. Thus every neighborhood of $g$ meets $P_\infty$, proving density.
\end{proof}

\subsection{Closure of $B_\infty$}

\begin{theorem}\label{thm:Binfty-closure-bar}
The subgroup $B_\infty$ is \emph{not} dense in $B^{lf}(\infty)$. More precisely,
\[
\overline{B_\infty}=\pi_*^{-1}\bigl(\Sym_f(\N)\bigr),
\]
and $\pi_*^{-1}(\Sym_f(\N))$ is a clopen subgroup of $B^{lf}(\infty)$.
\end{theorem}

\begin{proof}
Since $\Aut(\N)$ is discrete and $\pi_*$ is continuous, the preimage of any subset of $\Aut(\N)$ is clopen.
In particular, $\pi_*^{-1}(\Sym_f(\N))$ is clopen.

\smallskip\noindent
\emph{Step 1: $\overline{B_\infty}\subseteq \pi_*^{-1}(\Sym_f(\N))$.}
By Lemma~\ref{lem:image-Binfty-bar}, $B_\infty\subseteq \pi_*^{-1}(\Sym_f(\N))$, and the latter is closed, so it contains
$\overline{B_\infty}$.

\smallskip\noindent
\emph{Step 2: $\pi_*^{-1}(\Sym_f(\N))\subseteq \overline{B_\infty}$.}
Let $K:=\ker(\pi_*)\le B^{lf}(\infty)$.
By Corollary~\ref{cor:open-complete-subgroup}, $K$ is (as a topological group) identified with $H^{lf}(\infty)$, and under
this identification the subgroup $P_\infty\le H^{lf}(\infty)$ corresponds to the subgroup $P_\infty\le K$ defined above.
By Theorem~\ref{thm:P-infty-dense-bar}, $P_\infty$ is dense in $K$.

Now fix $\sigma\in\Sym_f(\N)$. Choose $b_\sigma\in B_\infty$ with $\pi_*(b_\sigma)=\sigma$
(Lemma~\ref{lem:image-Binfty-bar}). Then the fiber over $\sigma$ is the coset $\pi_*^{-1}(\sigma)=b_\sigma K$.
Left translation by $b_\sigma$ is a homeomorphism (Theorem~\ref{thm:B-complete-top-group}), hence $b_\sigma P_\infty$ is
dense in $b_\sigma K$.

Moreover, $b_\sigma P_\infty\subseteq B_\infty$: choose $n$ with $b_\sigma\in B_n$, and choose $m\ge n$ with
$x\in P_m^{\mathrm{img}}\subseteq K$; by Lemma~\ref{lem:directed-compatibility-bar} we have $B_n\subseteq B_m$, and
$P_m^{\mathrm{img}}\le B_m$, so $b_\sigma x\in B_m\subseteq B_\infty$.

Hence $\pi_*^{-1}(\sigma)=b_\sigma K\subseteq \overline{B_\infty}$ for every $\sigma\in\Sym_f(\N)$, and taking the union over
$\sigma\in\Sym_f(\N)$ yields $\pi_*^{-1}(\Sym_f(\N))\subseteq \overline{B_\infty}$.

\smallskip
Combining Steps 1 and 2 gives $\overline{B_\infty}=\pi_*^{-1}(\Sym_f(\N))$. Non-density follows because
$\Sym_f(\N)\neq \Aut(\N)$.
\end{proof}

\section{Ra\u{\i}kov completions and Polishness}\label{sec:universal}

\subsection{Ra\u{\i}kov completeness and a completion criterion}

\begin{definition}\label{def:raikov-complete}
Let $G$ be a Hausdorff topological group. The \emph{two-sided uniformity} on $G$ is the uniformity generated by
entourages
\[
E_U:=\{(g,h)\in G\times G:\ g^{-1}h\in U\ \text{and}\ gh^{-1}\in U\},
\]
where $U$ ranges over symmetric neighborhoods of the identity.
The group $G$ is \emph{Ra\u{\i}kov complete} if it is complete for the two-sided uniformity.
Its \emph{Ra\u{\i}kov completion} $\widehat G$ is the completion for this uniformity, characterized by the usual
universal property.
\end{definition}

There are some nice properties of Ra\u{i}kov complete groups. We list some of them.

\begin{itemize}
\item \textbf{Existence and universality of completion.}
Every topological group $G$ admits a Ra\u{\i}kov completion $\widehat G$
and the canonical homomorphism $i:G\to\widehat G$ is a topological embedding with dense image.
Moreover, for any continuous homomorphism $f:G\to H$ into a Ra\u{\i}kov complete group $H$,
there exists a unique continuous homomorphism $\widehat f:\widehat G\to H$ with
$f=\widehat f\circ i$.

\item \textbf{Closed subgroups inherit completeness.}
If $G$ is Ra\u{\i}kov complete and $H\le G$ is a closed subgroup, then $H$ is Ra\u{\i}kov complete.

\item \textbf{Products preserve completeness.}
A product $\prod_{i\in I} G_i$ is Ra\u{\i}kov complete if and only if each $G_i$ is Ra\u{\i}kov complete.

\item \textbf{Intrinsic characterization via completion.}
$G$ is Ra\u{\i}kov complete if and only if the canonical embedding $G\hookrightarrow\widehat G$
identifies $G$ as a closed subgroup of its Ra\u{\i}kov completion (equivalently, $G=\widehat G$).
\end{itemize}

By \cite[Theorem 4.3.26]{EngelkingGT}, a topological space is completely metrizable if and only
if it is a \v{C}ech-complete metrizable space. By \cite[Theorem 4.3.15]{ArhangelskiiTkachenko08},
a feathered group is \v{C}ech-complete if and only if it is Ra\u{i}kov complete. Since
metrizable topological groups are feathered, we have the following result.

\begin{theorem}\label{thm:AT-cor4327}
If $G$ is a \emph{metrizable} topological group, then $G$ is Ra\u{\i}kov complete if and only if $G$ is
\emph{completely metrizable}.
\end{theorem}

\subsection{The pure case: $H^{lf}(\infty)$ as the completion of $P_\infty$}

\begin{theorem}\label{thm:raikov-H}
The topological group $H^{lf}(\infty)$ is Ra\u{\i}kov complete. In particular, $H^{lf}(\infty)$ is the Ra\u{\i}kov completion of
$P_\infty$.
\end{theorem}

\begin{proof}
By Theorem~\ref{thm:metric-complete}, $H^{lf}(\infty)$ is completely metrizable and by Theorem~\ref{thm:AT-cor4327}, it is
Ra\u{\i}kov complete. Density of $P_\infty$ is from Theorem~\ref{thm:P-infty-dense-bar}. By \cite[Corollary 3.6.16]{ArhangelskiiTkachenko08}, $H^{lf}(\infty)$ is the Ra\u{\i}kov completion of $P_\infty$.
\end{proof}

\subsection{The finitary full braid case: completion of $B_\infty$}

\begin{definition}\label{def:Bfin}
Set
\[
B^{lf}_{\mathrm{fin}}(\infty)\ :=\ \pi_*^{-1}\bigl(\Sym_f(\N)\bigr)\ \le\ B^{lf}(\infty).
\]
This is clopen by Theorem~\ref{thm:Binfty-closure-bar}.
\end{definition}

\begin{theorem}\label{thm:raikov-Bfin}
The inclusion $B_\infty\hookrightarrow B^{lf}_{\mathrm{fin}}(\infty)$ is a topological embedding with dense image.
Moreover $B^{lf}_{\mathrm{fin}}(\infty)$ is Ra\u{\i}kov complete. Hence $B^{lf}_{\mathrm{fin}}(\infty)$ is the
Ra\u{\i}kov completion of $B_\infty$.
\end{theorem}

\begin{proof}
Density follows from Theorem~\ref{thm:Binfty-closure-bar}. By Theorem~\ref{thm:B-complete-top-group}, $B^{lf}(\infty)$ is completely metrizable; hence its clopen subgroup $B^{lf}_{\mathrm{fin}}(\infty)$ is completely metrizable as well. It is Ra\u{\i}kov complete by
Theorem~\ref{thm:AT-cor4327}. Again by \cite[Corollary 3.6.16]{ArhangelskiiTkachenko08}, $B^{lf}_{\mathrm{fin}}(\infty)$ is the
Ra\u{\i}kov completion of $B_\infty$.
\end{proof}

\begin{proposition}\label{prop:raikov-B}
The locally finite braid group $B^{lf}(\infty)$ is Ra\u{\i}kov complete.
\end{proposition}

\begin{proof}
By Theorem~\ref{thm:B-complete-top-group}, $B^{lf}(\infty)$ is completely metrizable, hence Ra\u{\i}kov complete by
Theorem~\ref{thm:AT-cor4327}.
\end{proof}

\subsection{Polishness}

\begin{definition}\label{def:polish}
A topological space is \emph{Polish} if it is separable and completely metrizable.
\end{definition}

\begin{theorem}\label{thm:H-polish}
The pure locally finite braid group $H^{lf}(\infty)$ is a Polish topological group.
\end{theorem}

\begin{proof}
By Theorem~\ref{thm:metric-complete}, $H^{lf}(\infty)$ is completely metrizable.
By Theorem~\ref{thm:P-infty-dense-bar}, it has a countable dense subgroup $P_\infty$, hence it is separable.
\end{proof}

\begin{theorem}\label{thm:B-not-polish}
The locally finite braid group $B^{lf}(\infty)$ is completely metrizable as a topological group,
but it is not Polish. Its subgroup $B^{lf}_{\mathrm{fin}}(\infty)$ is a Polish topological group.
\end{theorem}

\begin{proof}
\noindent\emph{Non-separability of $B^{lf}(\infty)$.}
The homomorphism $\pi_*:B^{lf}(\infty)\twoheadrightarrow \Aut(\N)$ is continuous and $\Aut(\N)$ is discrete, hence each fiber
$\pi_*^{-1}(\sigma)$ is clopen.
Thus $B^{lf}(\infty)$ contains the family of pairwise disjoint nonempty open subsets $\{\pi_*^{-1}(\sigma)\}_{\sigma\in\Aut(\N)}$.
Since $\Aut(\N)$ is uncountable, there are uncountably many such sets.

A separable space cannot contain uncountably many pairwise disjoint nonempty open subsets: if $D$ is countable dense, then each
nonempty open set meets $D$, giving an injection of the family into $D$.
Hence $B^{lf}(\infty)$ is not separable and therefore not Polish.

\smallskip
\noindent\emph{Polishness of $B^{lf}_{\mathrm{fin}}(\infty)$.}
By Theorem~\ref{thm:Binfty-closure-bar}, $B^{lf}_{\mathrm{fin}}(\infty)=\pi_*^{-1}(\Sym_f(\N))$ is clopen in $B^{lf}(\infty)$,
hence completely metrizable.
Also $\Sym_f(\N)$ is countable and each fiber $\pi_*^{-1}(\sigma)$ ($\sigma\in\Sym_f(\N)$) is clopen and homeomorphic to
$K=\ker(\pi_*)\cong H^{lf}(\infty)$, which is Polish by Theorem~\ref{thm:H-polish}.
Choose for each $\sigma\in\Sym_f(\N)$ a countable dense subset $D_\sigma\subset \pi_*^{-1}(\sigma)$; then
$\bigcup_{\sigma\in\Sym_f(\N)}D_\sigma$ is countable and dense in $B^{lf}_{\mathrm{fin}}(\infty)$ (because each fiber is open),
so $B^{lf}_{\mathrm{fin}}(\infty)$ is separable.
Therefore $B^{lf}_{\mathrm{fin}}(\infty)$ is Polish.
\end{proof}


\end{document}